\documentclass[a4paper,11pt]{amsart}
\usepackage[T1]{fontenc}
\usepackage[utf8]{inputenc}
\usepackage{amsmath}
\usepackage{url,graphicx}
\usepackage{amssymb, color}
\usepackage{amsmath, amsthm,  amsfonts}
\usepackage{mathrsfs}
\usepackage{cite}
\usepackage{enumerate}
\usepackage[none]{hyphenat}
\usepackage{colortbl}  
\usepackage{xcolor}
\usepackage{array}
\usepackage{amssymb}
\usepackage{esint}
\usepackage{appendix}

\newcommand{\intav}[1]{\mathchoice {\mathop{\vrule width 6pt height 3 pt depth  -2.5pt
\kern -8pt \intop}\nolimits_{\kern -6pt#1}} {\mathop{\vrule width
5pt height 3  pt depth -2.6pt \kern -6pt \intop}\nolimits_{#1}}
{\mathop{\vrule width 5pt height 3 pt depth -2.6pt \kern -6pt
\intop}\nolimits_{#1}} {\mathop{\vrule width 5pt height 3 pt depth
-2.6pt \kern -6pt \intop}\nolimits_{#1}}}

\theoremstyle{plain}
\newtheorem{theorem}{Theorem}[section]
\newtheorem{lemma}{Lemma}[section]
\newtheorem{proposition}{Proposition}[section]
\newtheorem{corollary}{Corollary}[section]
\newtheorem{definition}{Definition}[section]

\newtheorem{question}{Question}[section]
\theoremstyle{definition}
\newtheorem{remark}{Remark}[section]

\renewcommand{\theequation}{\arabic{section}.\arabic{equation}}
\numberwithin{equation}{section}
\numberwithin{figure}{section}
\numberwithin{remark}{section}

\allowdisplaybreaks

\usepackage{hyperref}
\usepackage{cite}

\hypersetup{hidelinks}

\begin{document}

\title{Regularity for fully nonlinear degenerate parabolic equations with strong absorption}

\author{Jiangwen Wang}
\address{School of Mathematics and Shing-Tung Yau Center, Southeast University, Nanjing 211189, P. R. China}
\email{jiangwen\_wang@seu.edu.cn}

\author{Jo\~{a}o Vitor da Silva}
\address{Departamento de Matem\'{a}tica, Universidade Estadual de Campinas (UNICAMP), Campinas-SP 13083-859, Brazil}
\email{jdasilva@unicamp.br}

\author{Feida Jiang$^*$}
\address{School of Mathematics and Shing-Tung Yau Center, Southeast University, Nanjing 211189, P. R. China; Shanghai Institute for Mathematics and Interdisciplinary Sciences, Shanghai 200433, P. R. China}
\email{jiangfeida@seu.edu.cn}

\date{\today}
\thanks{*Corresponding author. These authors contributed equally to this work.}

\begin{abstract}
In this paper, we investigate dead-core problems for fully nonlinear degenerate parabolic equations with strong absorption,
\begin{equation*}
    |Du|^{p} F(D^{2}u) - u_{t} = \lambda_{0}(x,t)\, u^{\mu}\, \chi_{\{u>0\}}(x,t)
    \qquad \text{in } \quad \Omega_{T} := \Omega \times (0,T),
\end{equation*}
where $0\le p<\infty$, $0<\mu<1$ and $u \geq 0 $ is a 
bounded viscosity solution satisfying $0\le u_t\le A_0u^\mu$ in the viscosity sense for some $0\le A_0<\infty$. We establish a sharp and improved parabolic $C^{\alpha}$-regularity estimate along the free boundary $\partial \{ u > 0 \}$, where
\[
\alpha := \frac{2+p}{1+p-\mu} > 1 + \frac{1}{1+p}.
\]
Moreover, we establish weak geometric properties of solutions, such as non-degeneracy and uniform positive density. As an application, we obtain a Liouville-type theorem for entire solutions and gradient bounds. Finally, as a byproduct of our approach, we derive a new formulation of the gradient decay property and establish a Harnack-type inequality on time slices. It is worth noting that the results presented here extend those in da Silva {\em et al.}(Pacific J. Math., \textbf{300} (2019), 179--213) and (J. Differential Equations., \textbf{264} (2018), 7270--7293) to the degenerate setting, and can be viewed as a parabolic analogue of da Silva {\em et al.} (Math. Nachr., \textbf{294} (2021), 38--55) and Teixeira (Math. Ann., \textbf{364} (2016), 1121--1134). 
\end{abstract}

\maketitle

\section{Introduction}\label{Section 1}

In this work, we study geometric regularity estimates for degenerate fully nonlinear parabolic equations of the form
\begin{equation}\label{DCP}
    |Du|^{p} F(D^{2}u) - u_{t} = \lambda_{0}(x,t)\, u^{\mu}\chi_{\{u>0\}}(x,t)
    \quad \text{in} \quad \Omega_{T} := \Omega \times (0,T), \tag{{\bf DCP}}
\end{equation}
subject to appropriate boundary conditions, where $0 \leq p < \infty$, $0 < \mu < 1$, $T>0$, and $\Omega \subset \mathbb{R}^{n}$ is a bounded smooth domain.
The coefficient  $\lambda_{0}(x,t)$  is a positive continuous weight function bounded above and below by positive constants; that is,
\[
\lambda_{0}\in C^{0}(\Omega_{T})
\qquad\text{and}\qquad
0<\mathcal{M}_{1}\leq \lambda_{0}(x,t)\leq \mathcal{M}_{2}<\infty.
\]

The operator $F : \mathrm{Sym}(n) \to \mathbb{R}$ satisfies the following structural conditions:

\vspace{1mm}

\noindent
\label{F1} {\bf (F1).} $F$ is uniformly elliptic with $F(\mathrm{O}_n) = 0$; namely, there exist constants $0 < \lambda \leq \Lambda$ such that, for all $\mathrm{M}, \mathrm{N} \in \mathrm{Sym}(n)$,
\[
  \mathscr{P}^{-}_{\lambda,\Lambda}(\mathrm{M}-\mathrm{N}) \leq F(\mathrm{M}) - F(\mathrm{N})
  \leq \mathscr{P}^{+}_{\lambda,\Lambda}(\mathrm{M}-\mathrm{N}),
\]
where $\mathscr{P}^{\pm}$ denote the \textit{Pucci extremal operators} \cite{CC95}, i.e.,
\[
  \mathscr{P}^{+}_{\lambda,\Lambda}(\mathrm{X})
  := \lambda \sum_{e_{i}<0} e_{i}(\mathrm{X}) + \Lambda \sum_{e_{i}>0} e_{i}(\mathrm{X}),
  \quad \text{and} \quad
  \mathscr{P}^{-}_{\lambda,\Lambda}(\mathrm{X})
  := \lambda \sum_{e_{i}>0} e_{i}(\mathrm{X}) + \Lambda \sum_{e_{i}<0} e_{i}(\mathrm{X});
\]

\vspace{1mm}

\noindent
\label{F2} {\bf (F2).} $F$ is convex or concave;

\vspace{1mm}

In addition, we also impose the following additional hypothesis,

\vspace{1mm}

{{
\noindent
\label{AC} {\bf (AC).} $ 0 \leq u_{t} \leq A_{0} u^{\mu } $, $ 0 \leq A_{0} < \infty $, in the viscosity sense.} }

\vspace{2mm}

Over the past decades, various classes of parabolic equations have been employed to describe phenomena arising in chemical reactions, physical processes, and biological systems. A central theme concerns reaction-diffusion processes exhibiting single-phase transitions, where the existence of nonnegative solutions plays a crucial role; see, for instance, \cite{BS84} and references therein. A prototypical example is the isothermal catalytic reaction-diffusion model (see \cite{Stakgold1986})
\begin{equation*}
\left\{
     \begin{aligned}
     & u_{t} = \Delta u - f(u) && \text{in } \quad \Omega_{T}, \\
     & u(x,t) = g(x,t) &&  \text{on } \quad  \partial \Omega \times (0,T), \\
     & u(x,0) = u_{0}(x) && \text{in }  \quad   \overline{\Omega},
     \end{aligned}
\right.
\end{equation*}
with the data satisfying
\[
  0 < u_{0} \in C^{0}(\overline{\Omega}), \quad g(x,t) = \mathcal{G} > 0, \quad
  u(x,0) = \mathcal{G} \quad \text{ for all }    x \in \partial \Omega,
\]
where $\Omega \subset \mathbb{R}^{n}$ is a smooth, bounded domain. In this context, \(u\) represents the concentration of the reactant.
The reactant becomes inactive (forming a dead core) when \(u\) reaches zero. The boundary condition indicates that the reactant is supplied at the boundary with a fixed concentration.

In such evolution problems, if $f$ is locally Lipschitz continuous, satisfies $f(s) > 0$ for $s > 0$ and $f(0) = 0$, then the maximum principle guarantees that nonnegative solutions remain strictly positive. However, when $f$ fails to be Lipschitz at the origin (e.g., $f(s) \simeq s^{\mu}$ for $\mu \in (0, 1)$), nonnegative solutions may develop a {\it plateau region}, a subset of positive measure where $u$ vanishes identically-commonly referred to as a {\em dead-core set} \cite{CW03}. Additionally, the reaction exhibits strong absorption, since the reaction rate
\(-f(u)/u \to -\infty\) as \(u \to 0^{+}\).

\vspace{2mm}

The study of parabolic equations often focuses on the temporal asymptotic behavior of solutions near singular points. Among the key phenomena are blow-up (studied for the model $u_t = \Delta u + u^{\mu}$, for $\mu>1$), quenching (studied for the model $u_t = \Delta u - u^{-\mu}$, for $0<\mu$), extinction, and dead-core formation  (studied for the model $u_t = \Delta u - u^{\mu}$, for $0<\mu<1$). Unlike quenching, where certain derivatives blow up, dead-core solutions remain regular but reach zero in finite time (see \cite{Guo2013} for related topics).

Concerning modern trends, for the classical model in divergence form,
\begin{equation}\label{Sec:eq1}
  \Delta u - u_{t} = \lambda_{0} u^{\mu}\chi_{\{u>0\}}(x,t)
  \quad \text{in} \quad \Omega_{T},
\end{equation}
with $ 0 < \mu < 1$,
Choe-Weiss \cite{CW03} established several qualitative and quantitative properties of its solutions using variational methods. Later, da Silva {\it et al.} \cite{SOS18} derived geometric regularity estimates for degenerate parabolic equations of $q$-Laplacian type ($2 \leq q < \infty$) with strong absorption,
\begin{equation}\label{Sec:eq2}
  \Delta_{q}u - u_{t} = \lambda_{0} u^{\mu}\chi_{\{u>0\}}(x,t)
  \quad \text{in} \quad \Omega_{T},
\end{equation}
where $0 < \mu < 1$. Nevertheless, the techniques in \cite{CW03} and \cite{SOS18} cannot be directly applied to non-divergence type equations due to the absence of a variational framework and related key devices. The non-variational fully nonlinear approach to dead-core problems was introduced in the elliptic setting by Teixeira \cite{T16} and later extended to the parabolic setting by da Silva and Ochoa \cite{SO19}. While the present work builds on the geometric tangential analysis and flatness improvement techniques developed in these foundational works, the presence of the degenerate factor $|Du|^p$ in \eqref{DCP} gives rise to substantial new technical difficulties. In contrast to the non-degenerate settings considered in \cite{T16, SO19}, the degeneracy of our model requires a tailored intrinsic parabolic scaling and the construction of new, more delicate barrier functions to capture the interplay between the operator's degeneracy and the strong absorption term.

\vspace{2mm}

Motivated by these works \cite{CW03}, \cite{SO19}, and \cite{SOS18}, our main objective is to investigate the geometric and analytic properties of dead-core solutions to \eqref{DCP}. To the best of our knowledge, sharp and improved regularity estimates for fully nonlinear degenerate parabolic dead-core problems have not yet been thoroughly addressed. In this sense, the results presented here can be regarded as a parabolic counterpart of \cite{SLR21} and \cite{T16} and as an extension of the analysis in \cite{SO19} and \cite{SOS18}. Indeed, as a preliminary motivation, recall the following non-divergence representation of the $q$-Laplacian:
\[
\Delta_q u
= |Du|^{\,q-2}\,\mathrm{Tr}\!\left[\left(\textbf{I}{\rm{d}_{n} } + (q-2)\frac{Du \otimes Du}{|Du|^2}\right) D^2u\right]
= \mathcal{G}_q(Du, D^2u),
\]
valid for every $q \ge 2$. This is a fully nonlinear operator with a degenerate structure, closely aligned with the model introduced in~\eqref{DCP}, which supports our choice to study such degenerate equations.

Therefore, our manuscript addresses that understanding these dead-core phenomena in a non-variational degenerate scenario is essential not only for fundamental regularity theory but also for possible applications in phenomena such as chemical reactions, heat transfer, and other nonlinear processes (cf. \cite{SO19}, \cite{SOS18}, \cite{Guo2013}, and \cite{Stakgold1986}).

\subsection{Main results}
We now establish the first main result, concerning sharp and improved regularity of solutions near the free boundary, which reads as follows.

\begin{theorem}[{\it Improved regularity along free boundary}]
\label{Thm1}
Suppose that $ F $ satisfies \hyperref[F1]{\bf (F1)}. Let $ u $ be a nonnegative and bounded viscosity solution to \eqref{DCP}, {so that \hyperref[AC]{\bf (AC)} holds in $ \Omega_{T} $}. For every compact set $ K \Subset  \Omega_{T}  $, there exists a constant $ \mathrm{C}_0 > 0 $, depending only on $ n, \lambda, \Lambda, \mathcal{M}_{1}, \mathcal{M}_{2}, p, \mu, \|u\|_{L^{\infty}(\Omega_{T})}, A_{0}$ and $ \mathrm{dist}(K, \partial_{\mathrm{par}} \Omega_{T}) $ such that, for every $ (x_{0}, t_{0}) \in \partial \{u>0\} \cap K   $,
\begin{equation*}
  u(x,t) \leq \mathrm{C}_0  \mathrm{dist}_{p}((x,t), (x_{0}, t_{0}))^{\frac{2+p}{1+p-\mu}},
\end{equation*}
for all $ (x,t) $ sufficiently close to $ (x_{0}, t_{0}) $, where
\begin{equation*}
  \mathrm{dist}_{p}((x,t), (x_{0}, t_{0})) = |x-x_{0}| + |t-t_{0}|^{\frac{1}{\theta}}, \quad{and} \quad \theta =  \frac{(2+p)(1-\mu)}{1+p-\mu}.
\end{equation*}
\end{theorem}

We also give a lower growth estimate for dead-core solutions of \eqref{DCP}.

\begin{theorem}[{\it Non-degeneracy}]
\label{Thm2}
Suppose $ F $ satisfies \hyperref[F1]{\bf (F1)}. Let $ u $ be a viscosity solution to \eqref{DCP}. Then for every $ (x_{0}, t_{0}) \in \overline{\{u>0\}} $ and $ Q_{r}(x_{0}, t_{0}) \Subset  \Omega_{T} $, there holds
\begin{equation*}
  \sup_{\partial_{\mathrm{par}} Q_{r}^{-}(x_{0}, t_{0})} u(x,t) \geq \mathrm{C}_{0}^{*} r^{\frac{2+p}{1+p-\mu}}
\end{equation*}
for a positive constant $ \mathrm{C}_{0}^{*} = \mathrm{C}_{0}^{*}(n, \lambda, \Lambda, \mathcal{M}_{1}, \mathcal{M}_{2}, p, \mu) $.
\end{theorem}

Before proceeding further, we make the following heuristic remarks.

\begin{remark}
The proof of Theorem \ref{Thm1} relies on the construction of a suitable barrier function of the form
\[
\Phi_r(x,t)
=
\mathcal{C}
\left(
|x|^{\frac{2+p}{1+p}}
+
\mathcal{B}
(t+r^\theta)^{
\frac{1+p-\mu}{(1+p)(1-\mu)}
}
\right)^{\frac{1+p}{1+p-\mu}},
\]
where $\mathcal{C},\mathcal{B}>0$ are universal constants, combined with the comparison principle in
Lemma \ref{Se2:lem1}.   

Under the remaining assumptions on the parameters,
the restriction $0<\mu<1$ ensures that the exponent
of $t+r^\theta$ is positive, a property used in the
forward comparison argument. Thus, the present barrier
construction requires $0<\mu<1$, rather than the broader
range $0<\mu<1+p$. This requirement accounts for the
distinction between our parameter range and those
considered in \cite{SLR21,T16}.   
\end{remark}


\begin{remark}
Note that, when $0<\mu<1$, \eqref{DCP} admits nontrivial
nonnegative solutions that vanish on open sets.
To illustrate this phenomenon, assume that $\lambda_0>0$
is constant. Since $F(\mathbf{O}_n)=0$, the time-dependent function
\[
u(t)=\big[(1-\mu)\lambda_0(t_0-t)\big]_+^{\frac{1}{1-\mu}}
\]
is a viscosity solution of~\eqref{DCP}.
It satisfies $u_t=-\lambda_0u^\mu$ and therefore does not
satisfy \hyperref[AC]{\textbf{(AC)}} wherever $u>0$.

For $F(\mathrm{X})=\operatorname{tr}(\mathrm{X})$, stationary viscosity
solutions are given by
\[
u(x,t)=
\mathrm{C}_{p,\mu,\lambda_0}
(x_i)_\pm^{\frac{2+p}{1+p-\mu}},
\quad i=1,\ldots,n,
\]
where $(x_i)_\pm:=\max\{\pm x_i,0\}$ and
\[
\mathrm{C}_{p,\mu,\lambda_0}
=
\left[
\frac{\lambda_0(1+p-\mu)^{2+p}}
{(2+p)^{1+p}(1+\mu)}
\right]^{\frac{1}{1+p-\mu}}.
\]
These stationary solutions satisfy \hyperref[AC]{\textbf{(AC)}} with $A_0=0$ and have a half-space as their zero set.
\end{remark}


\begin{remark}
Under the assumptions of Theorem~\ref{Thm1}, for every
$(x_0,t_0)\in\partial\{u>0\}\cap\Omega_T$, we have
\[
\sup_{Q_r(x_0,t_0)}u
\leq \mathrm{C}r^{\frac{2+p}{1+p-\mu}}
\]
for all sufficiently small $r>0$, where
\[
0<\theta
=\frac{(2+p)(1-\mu)}{1+p-\mu}
=1+\frac{1-\mu(p+1)}{1+p-\mu}
\leq2.
\]
This provides a pointwise upper growth estimate at
free boundary points with respect to the intrinsic scaling.

We point out that for $p>0$,
\cite[Theorem~1.3]{LLYZ25} establishes spatial
$C^{1,1/(1+p)}$ regularity for the corresponding homogeneous
equation under assumptions
\hyperref[F1]{\textbf{(F1)}} and
\hyperref[F2]{\textbf{(F2)}}. Since
\[
\frac{2+p}{1+p-\mu}>1+\frac{1}{1+p},
\]
Theorem~\ref{Thm1} gives a higher pointwise order of
vanishing at the free boundary. Moreover, if $0\leq p<2\mu$, then
$\frac{2+p}{1+p-\mu}>2$, and hence   
\[
u(x,t_0)=o(|x-x_0|^2)
\quad\text{as }x\to x_0.
\]
Thus $u(\cdot,t_0)$ has a zero second-order spatial
Peano expansion at $x_0$.
\end{remark}

\begin{remark}
For the singular case $ -1 < p < 0 $, the main challenge in obtaining regularity estimates from Theorem \ref{Thm1} is that the methods here do not work for Lemma \ref{Se3:lem1}. For this reason, the treatment of the singular setting requires the development of new approaches and modern techniques.
\end{remark}

\begin{remark}
The upper growth estimate in Theorem~\ref{Thm1} also extends to
\[
|Du|^pF(D^2u)-u_t=f(u), \quad  \text{in} \quad  \Omega_{T}:= \Omega \times (0,T)
\]
provided that condition \hyperref[AC]{\textbf{(AC)}} holds and
$f\in C^{0}([0,\infty))$ is nonnegative, nondecreasing, and satisfies
$f(0)=0$ and $ 0\le f(s)\le \mathrm{C}_*s^\mu $ for all sufficiently small $s\ge0$, where $\mathrm{C}_*>0$. Examples include $f(s)=e^s-1$ and $f(s)=\log(1+s^2)$, see \cite[Section 7]{DT20} and \cite[Section 3]{SO19}.
In addition, if there exist $c_*>0$ and $s_*>0$ such that $ f(s)\ge c_*s^\mu $ for $ 0\le s\le s_* $, then the barrier argument in Theorem~\ref{Thm2} also yields the corresponding nondegeneracy estimate for all sufficiently small radii.
\end{remark}

Next, a fundamental consequence of Theorem \ref{Thm1} is stated as follows:

\begin{theorem}[{\it Gradient decay}]\label{Thm3}
Suppose that $ F $ satisfies \hyperref[F1]{\bf (F1)}, \hyperref[F2]{\bf (F2)} and $ 0 < \mu \leq \frac{1}{1+p} $. Let $u$ be a bounded, nonnegative viscosity solution of \eqref{DCP} in $Q_1$, and assume that \hyperref[AC]{\textbf{(AC)}} holds. For any $(z,s) \in \{u>0\} \cap Q_{1/2}$, there holds
\begin{equation*}
|Du(z,s)| \leq \mathrm{C}\,\mathrm{dist}_{p}\big((z,s),\partial\{u>0\}\big)^{\frac{1+\mu}{1+p-\mu}},
\end{equation*}
where $ \mathrm{C} >0 $ is a constant, depending only on $ n, \lambda, \Lambda, \mathcal{M}_{1}, \mathcal{M}_{2}, p, \mu, A_0$ and $ ||u||_{L^{\infty}(Q_1)}$.
\end{theorem}

\vspace{1mm}

In the final result, we shall be devoted to establishing a general Liouville-type theorem for any entire viscosity solutions to \eqref{DCP}.

\begin{theorem}[{\it Liouville-type result}]\label{Thm4}
Suppose that $ F $ satisfies \hyperref[F1]{\bf (F1)}. Let $ u $ be a nonnegative entire viscosity solution to
\begin{equation*}
  |Du|^{p} F(D^{2}u) - \partial_{t} u = \lambda_{0}(x,t) u^{\mu}\chi_{\{u>0\}}(x,t)
\end{equation*}
with $ \lambda_{0} $ as before. If $ u(x,t) = \text{o} \big(\max\{|x|, |t|^{\frac{1}{\theta}}\}^{\frac{2+p}{1+p-\mu}}\big)  $ as $ \max\{|x|, |t|^{\frac{1}{\theta}}\} \rightarrow \infty $, then
\begin{equation*}
u \equiv 0  \ \ \text{in}  \ \  \mathbb{R}^{n} \times \mathbb{R}.
\end{equation*}
\end{theorem}

\subsection{Ideas of the proof}
From a mathematical standpoint, this article is, to the best of our knowledge, the first to establish sharp and improved $C^{\alpha}$ regularity estimates across the free boundary $\partial\{u>0\}$, together with weak geometric properties of viscosity dead-core solutions $u$. The proof of Theorem~\ref{Thm1} is inspired by techniques from the regularity theory of fully nonlinear equations and free boundary problems (see \cite{CC95, 1W92, 2W92, 3W92, Sh03, SLR21, SOS18, SO19, ST17}).

{We first establish the normalized decay estimate in
Lemma~\ref{Se3:lem1} by contradiction, using intrinsic
scaling, temporal estimates, and elliptic compactness.
We then combine a dyadic iteration with the comparison
argument in Proposition~\ref{Section3:prop31} to derive
the normalized upper growth estimate in
Theorem~\ref{Sec3:thm1}. Finally, we normalize the solution
in amplitude, space, and time and apply this estimate.
Rescaling back to the original variables yields
Theorem~\ref{Thm1}.

The proof involves several technical difficulties,
which we address as follows.

\begin{enumerate}[i)]
\item
Amplitude normalization alters the balance between
the diffusion and time-derivative terms. We therefore
use an \emph{intrinsic rescaling} to preserve the
structural form of the equation. If the decay estimate
fails, we select a sequence of solutions $u_k$ and
dyadic scales $2^{-j_k}$ and set
\[
v_k(x,t)
:=
\frac{u_k(2^{-j_k}x,\alpha_k t)}
{L_{2^{-(j_k+1)}}[u_k]},
\qquad (x,t)\in Q_1^-,
\]
where the temporal scaling factors $\alpha_k$ are
defined in Section~\ref{Section 3}.
The resulting sequence is uniformly bounded and satisfies
\[
v_k(0,0)=0,
\quad \quad \text{and} \quad 
\sup_{x\in B_{1/2}}v_k(x,0)=1.
\]
The rescaled operators retain the ellipticity constants
of $F$, and the rescaled absorption coefficients
converge to zero.

\item
A further difficulty is to establish compactness under
the structural assumption
\hyperref[F1]{\textbf{(F1)}} and the temporal condition
\hyperref[AC]{\textbf{(AC)}}.
We address this by considering fixed-time slices and
applying elliptic regularity estimates.
Specifically, setting $\omega_k(x):=v_k(x,0)$,
Lemmas~\ref{Section 2:lemma2} and~\ref{Section 2:lemma3}
yield
\[
0
\le
|D\omega_k|^p\widetilde F_k(D^2\omega_k)
\le h_k
\qquad \text{in } B_1
\]
in the viscosity sense, where $h_k\to0$.
Together with the uniform boundedness of $\omega_k$,
the interior H\"older estimate in
Proposition~\ref{prop:IS-holder} implies local
equicontinuity. The Arzel\`a--Ascoli theorem then yields
a subsequence converging locally uniformly in $B_1$
to a function $\omega_0\in C^0(B_1)$.

To control the temporal dependence, we combine the
rescaled condition \hyperref[AC]{\textbf{(AC)}} with
Lemma~\ref{Section 2:lemma2} to obtain
\[
|v_k(x,t)-v_k(x,s)|
\le \mathrm{C}\varepsilon_k|t-s|,
\quad \text{as} \quad \varepsilon_k\to0,
\]
with $\mathrm{C}>0$ independent of $k$.
Combining this vanishing time oscillation with the local
uniform convergence $\omega_k\to\omega_0$, we conclude
that $v_k$ converges locally uniformly to the stationary
limit $\omega_0$.

\item
A further difficulty is that the limiting equation
retains the degenerate factor $|D\omega_0|^p$.
Passing to a further subsequence, we obtain local
uniform convergence of the rescaled operators to
$\widetilde F_0$. Stability under locally uniform
convergence then yields
\[
|D\omega_0|^p\widetilde F_0(D^2\omega_0)=0
\qquad \text{in } B_1
\]
in the viscosity sense.
The cutting lemma \cite[Lemma~6]{CL13} implies that
\[
\widetilde F_0(D^2\omega_0)=0
\qquad \text{in } B_1.
\]
Since $\omega_0\ge0$ and $\omega_0(0)=0$, the strong
minimum principle gives $\omega_0\equiv0$ in $B_1$.
This contradicts the normalization
\[
\sup_{B_{1/2}}\omega_0=1
\]
and completes the proof of
Lemma~\ref{Se3:lem1}.

\item
Finally, the dyadic iteration yields growth estimates
in backward parabolic cylinders, whereas the desired
estimate must also hold forward in time from the
reference free-boundary point.
To establish the forward estimate, we construct a barrier
adapted to the \emph{intrinsic scaling}, following the
comparison arguments in \cite{SO19,SOS18,Sh03}.

The main difficulty is to ensure that the barrier
satisfies the required differential inequality while
remaining above $u$ on the parabolic boundary.
We address this by introducing the time shift
$t+r^\theta$ and choosing the barrier parameters
in a suitable order.
The shift provides the required control on the initial
time slice, while condition
\hyperref[AC]{\textbf{(AC)}} yields the bound needed
on the lateral boundary.
The comparison principle then gives the forward
growth estimate.
Combining the forward and backward estimates and
applying them at nearby free-boundary points yields
the normalized upper growth estimate in
Theorem~\ref{Sec3:thm1}.
\end{enumerate} 
}

{To prove the non-degeneracy estimate
(Theorem~\ref{Thm2}), we construct a supersolution that
vanishes at the center of a backward cylinder and is
bounded below by a positive multiple of
$r^{(2+p)/(1+p-\mu)}$ on its parabolic boundary.
When the center lies in the positive phase, the comparison
principle (Lemma~\ref{Se2:lem1}) forces $u$ to exceed this
supersolution somewhere on the parabolic boundary,
yielding the required lower bound.
Continuity then extends the estimate to the closure of
the positive phase.}    

\bigskip

\vspace{2mm}




 \subsection{Primary technical advancements} The primary objective of this work is to extend the non-variational geometric framework for dead-core problems, introduced in \cite{T16} and further developed in \cite{SO19}, to a genuinely degenerate parabolic setting. The main technical contributions of this work are summarized below:
 \begin{itemize}

 \vspace{1mm}

\item[\checkmark] We extend the geometric regularity theory to the regime
$0\le p<\infty$. Building on \cite{SO19}, we address the difficulties
caused by the factor $|Du|^p$ through an intrinsic
parabolic scaling that balances the degenerate diffusion
and the strong absorption;

 \vspace{1mm}

\item[\checkmark] {We combine the quantitative temporal control in \hyperref[AC]{\textbf{(AC)}} with elliptic H\"older estimates on fixed-time slices to establish compactness in the normalized contradiction argument.
The rescaled temporal estimate yields a stationary
limit, to which the cutting lemma and the strong
minimum principle apply.
To the best of our knowledge, this approach has not
previously been used to derive sharp free-boundary
growth estimates for fully nonlinear parabolic
dead-core equations with gradient degeneracy.}

\vspace{1mm}

\item[\checkmark] We construct specialized comparison functions adapted to the balance between the degenerate operator and the strong absorption term. These barriers require a careful treatment of the intrinsic scaling and constitute a substantial technical refinement of the non-variational methods developed for the non-degenerate case in \cite{SO19}.

\vspace{1mm}

\item[\checkmark] We further investigate the growth of solutions near free boundary points through a Harnack-type inequality (cf.~\cite[Theorem~8.3]{BJrDaSR2023}) for viscosity solutions of~\eqref{DCP} on $t$-slices of the parabolic domain. This yields a complementary regularity estimate; see Theorem~\ref{Sec6:thm6.1}.    
\end{itemize}

\subsection{Structure of the paper}
The remainder of the paper is organized as follows. Section~\ref{Section 2} introduces the notion of viscosity solutions to \eqref{DCP} and establishes several auxiliary lemmas, as well as the existence of a nonnegative viscosity solution. Section~\ref{Section 3} is devoted to the proofs of Theorems~\ref{Thm1} and~\ref{Thm2}. Section~\ref{Section 4} derives density consequences from these results. Finally, Section~\ref{Section 5} contains the proofs of Theorems~\ref{Thm3} and~\ref{Thm4}, together with the blow-up analysis.

In Section \ref{Section 6}, we derive a new formulation of the gradient estimate. Moreover, a Harnack-type inequality for viscosity solutions to \eqref{DCP} in $t-$slices of the parabolic domain is addressed in the spirit of clever observation. These results may be of independent interest (see Corollary~\ref{Sec6:Coro2} and Theorem~\ref{Sec6:thm6.1}). We also outline several promising directions for future research. Finally, \hyperref[Appendix:A]{Appendix A} and \hyperref[Appendix:B]{Appendix B} contain the proofs of Lemma \ref{Se2:lem1} and Lemmas \ref{Section 2:lemma2}--\ref{Section 2:lemma3}, respectively.

\subsection{Notations}
We summarize below the basic notation used throughout the paper.

\begin{itemize}

\item For $(x_{0}, t_{0}) \in \Omega \times \mathbb{R}$ and $r > 0$, let $B_{r}(x_{0})$ denote the open ball centered at $x_{0}$ with radius $r$. We define the following parabolic cylinders:
\begin{align*}
       & Q_{r}(x_{0}, t_{0}) := B_{r}(x_{0}) \times (t_{0} - r^{\theta},\, t_{0} + r^{\theta}), \\
       & Q_{r}^{+}(x_{0}, t_{0}) := B_{r}(x_{0}) \times [t_{0},\, t_{0} + r^{\theta}),  \\
       & Q_{r}^{-}(x_{0}, t_{0}) := B_{r}(x_{0}) \times (t_{0} - r^{\theta},\, t_{0}],
\end{align*}
where $\theta$ denotes the  intrinsic time-scaling exponent
\begin{equation*}
  \theta = \theta(p,\mu) := \frac{(2+p)(1-\mu)}{1+p-\mu}.
\end{equation*}
In particular, we denote $ Q_{r}:= B_{r} \times (-r^{\theta}, r^{\theta}) $.

\item We write $u_{t} = \partial_{t}u = \frac{\partial u}{\partial t}$, and define the parabolic distance by
\[
\mathrm{dist}_{p}((x,t), (y,s)) := |x - y| + |t - s|^{\frac{1}{\theta}}.
\]

\item For a parabolic domain $Q := \Omega \times I'$, its parabolic boundary is given by
\[
\partial_{\mathrm{par}} Q := (\overline{\Omega} \times \{a\}) \cup (\partial \Omega \times I'),
\]
where $I'$ is an interval $[a, b)$.

\item $\mathscr{L}^{n+1}(E)$ denotes the $(n + 1)$-dimensional Lebesgue measure of a measurable set $E$.

\item $\mathscr{H}^{n}(E)$ denotes the $n$-dimensional Hausdorff measure of a measurable set $E$.

\item $ \textbf{I}{\rm{d}_{n} } $ denotes the $n\times n$ identity matrix.

\item $ f \simeq g $ denotes there exists a constant $ \mathrm{C} > 0 $ such that $ \frac{1}{\mathrm{C}}g \leq f \leq \mathrm{C}g $.

\item $\mathrm{C}$ denotes a generic positive constant that may vary from line to line.
\end{itemize}

{\bf Acknowledgments}.The authors thank the anonymous referees for their constructive comments, which have substantially improved the manuscript. The first author gratefully acknowledges Professor Hyungsung Yun (Changwon National University, Republic of Korea) for insightful discussions on the regularity theory of fully nonlinear degenerate parabolic equations. F. Jiang has been supported by the National Natural Science Foundation of China (No. 12271093) and the Jiangsu Provincial Scientific Research Center of Applied Mathematics (Grant No. BK20233002), and the Shanghai Institute for Mathematics and Interdisciplinary Sciences (SIMIS) under grant number SIMIS-ID-2025-AD. J.V. da Silva has received partial support from CNPq-Brazil under Grant No. 307131/2022-0, and Chamada CNPq/MCTI No. 10/2023 - Faixa B - Consolidated Research Groups under Grant No. 420014/2023-3. J.V. da Silva has been supported by FAEPEX-UNICAMP (Project No. 2441/23, Special Calls - PIND - Individual Projects, 03/2023), and FAPESP-Brazil under the Grant No.  2025/09344-1 - Special Programs - Special Projects - First Projects - Call for Proposals (2025) - 1st Cycle.

\vspace{3mm}
\section{Preliminaries}\label{Section 2}
{{We recall the definition of viscosity solutions to
\eqref{DCP}, the comparison principle, and the gradient
regularity estimate in \cite[Theorem~1.3]{LLYZ25}.
Under \hyperref[AC]{\textbf{(AC)}}, we derive temporal
estimates and elliptic inequalities for fixed-time slices.
We also recall the stability of viscosity solutions,
the cutting lemma, and an interior H\"older estimate
needed for compactness.
Finally, we establish the existence of viscosity
solutions to the Dirichlet problem associated with
\eqref{DCP}.}}

\subsection{Viscosity solutions and auxiliary lemmas}

 For $p\geq 0$, the operator $\mathcal{L}_p[u]:= |Du|^{p}F(D^{2}u)-u_t$ is well-defined even when $Du=0$. We therefore adopt the standard notion of viscosity solutions; see \cite[Definition 3.4]{CL12} and \cite[Section 8]{CIL92}. \begin{definition}[{\it Viscosity convention}] \label{Pre:def1} We say that $u\in C^{0}(\Omega_T)$ is a viscosity subsolution (resp.~supersolution) of \begin{equation}\label{Section2:maindef} |Du|^{p}F(D^{2}u)-u_t=f(x,t,u) \quad\text{in }\Omega_T \end{equation} if, for any $(x_0,t_0)\in \Omega_T$ and every $\varphi\in C^{2,1}(\Omega_T)$ touching $u$ from above (resp.~below) at $(x_0,t_0)$, one has \begin{equation*} |D\varphi(x_0,t_0)|^{p}F(D^{2}\varphi(x_0,t_0)) -\varphi_t(x_0,t_0) \geq\;(\text{resp. }\leq)\; f(x_0,t_0,u(x_0,t_0)). \end{equation*} We say that $u$ is a viscosity solution of \eqref{Section2:maindef} in $\Omega_T$ if it is both a viscosity subsolution and a viscosity supersolution. \end{definition}

Now we shall present a comparison principle for a more general class of fully nonlinear degenerate parabolic equations of the form
\begin{equation}\label{G-DCP}
\widetilde{\Phi}(|Du|) F(D^{2}u) - \partial_{t} u - \lambda_{0}(x,t)f(u) = g(x,t)   \quad  \text{in} \quad  \Omega_{T}   \tag{{\bf G-DCP}}
\end{equation}
where $ f $ is continuous non-decreasing and $ g $ is  continuous. The degeneracy is governed by the general law
\begin{equation}\label{Se2:eq0}
  \widetilde{\Phi}(|Du|) \simeq |Du|^{p},  \quad  \widetilde{\Phi} \in C^{0}([0, \infty)),    
\end{equation}
with exponents $ p \geq 0 $.

\begin{lemma}[{\it Comparison principle}]
\label{Se2:lem1}
Assume that assumptions \hyperref[F1]{\bf (F1)} and \eqref{Se2:eq0} are satisfied. Suppose that $u,v\in C^{0}(\overline{\Omega}\times[0,T))$ are, respectively, a viscosity subsolution and a viscosity supersolution of \eqref{G-DCP}, with source terms $g$ and $\widetilde g$, respectively. Moreover, assume that $ \widetilde{g} \leq g $ in $ \Omega_{T} $. If $ u(x,t) \leq v(x,t) $ for any $ (x,t) \in \partial_{\mathrm{par}} \Omega_{T} $, then $ u \leq v $ in $ \Omega_{T} $.
\end{lemma}

We postpone its proof to \hyperref[Appendix:A]{Appendix A}.

In order to access the gradient decay of viscosity solutions (Theorem \ref{Thm3}), and the proof of Theorem~\ref{Thm5.1-B-L}, we need the following result:  
\begin{proposition}[{\it \cite[Theorem 1.3]{LLYZ25}}] \label{Se2:Pro1}
Assume that $ 0 \leq p< \infty$, $ F $ satisfies \hyperref[F1]{\bf (F1)} and  \hyperref[F2]{\bf (F2)}. Let $ u $ be a bounded viscosity solution of
\begin{equation}\label{Section2: eq11}
  |Du|^{p} F(D^{2}u)-u_{t} = 0 \quad \text{in} \quad  Q_{1}.
\end{equation}
Then there exists a constant $ \mathrm{C} > 0 $, depending only on $ n, p, \lambda, \Lambda, \mu  $ and $ \|u\|_{L^{\infty}(Q_{1})}$, such that
\begin{equation*}
|Du(x,t)-Du(y,s)| \leq \mathrm{C} (|x-y|^{\frac{1}{1+p}}+ |t-s|^{\frac{1}{2+p}})  \quad   \text{for all}  \quad   (x,t), (y,s) \in   \overline{Q_{1/2}},
\end{equation*}
and
\begin{equation*}
|u(x,t) - u(x,s)| \leq  \mathrm{C} |t-s|  \quad   \text{for all}  \quad   (x,t), (y,s) \in   \overline{Q_{1/2}}.
\end{equation*}
\end{proposition}

\begin{remark}\label{Sec2:rmk2} 
\begin{itemize} 
\item[(i)] \label{Section2:remark1} Following the reduction argument used in the proof of
\cite[Theorem~1.1]{A20}, and using
Proposition~\ref{Se2:Pro1} for the homogeneous equation,
we obtain interior spatial $C^{1,\alpha}$ estimates
for bounded viscosity solutions of
\[
|Du|^p F(D^2u)-u_t=f
\quad\text{in }Q_1,
\]
under assumptions \hyperref[F1]{\textbf{(F1)}} and
\hyperref[F2]{\textbf{(F2)}}, provided that
$f\in C^{0}(Q_1)\cap L^\infty(Q_1)$.
Here $ 0<\alpha<\frac{1}{1+p} $, and the local estimates
depend only on the structural constants,
$\|u\|_{L^\infty(Q_1)}$, $\|f\|_{L^\infty(Q_1)}$, and the distance to the parabolic boundary.
No differentiability assumption on $F$ is required.

\item[(ii)] A key improvement of Proposition~\ref{Se2:Pro1} over the previous result~\cite[Theorem 1.1]{LLY24} is the removal of the differentiability assumption on the operator $F$. Lee--Lee--Yun~\cite[Theorem 1.1]{LLY24} established $C^{1,\alpha}$ regularity, with $0<\alpha\leq\frac{1}{1+p}$, under the assumptions \hyperref[F1]{\bf (F1)}, \hyperref[F2]{\bf (F2)}, and $F\in C^{1,\kappa}$, with $0<\kappa\leq1$. This additional differentiability was used to control the constants in the small-perturbation estimates. By contrast, Proposition~\ref{Se2:Pro1} combines the Bernstein technique with a refined approximation scheme to obtain uniform estimates for the time derivative without any additional differentiability assumption on $F$. In particular, it recovers the optimal $C^{1,\frac{1}{1+p}}$ spatial regularity and yields Lipschitz continuity in time. This refined framework is useful in our rescaling arguments, since the $C^{1,\kappa}$ norms of the rescaled operators need not remain uniformly bounded.
\end{itemize} 
\end{remark}

{ The following result gives two consequences of
\hyperref[AC]{\textbf{(AC)}} that are essential to the
proof of the normalized decay lemma
(Lemma~\ref{Se3:lem1}). Specifically,

\begin{lemma}[{\it Time control}]\label{Section 2:lemma2}
Let $ u \geq 0 $ be continuous on $ B_{1} \times I $, where $ I $ is an open interval. We also assume that \hyperref[AC]{\bf (AC)} holds. For each fixed $x\in B_1$, the function $t\mapsto u(x,t)$ is
nondecreasing and locally Lipschitz. For any $ [a,b] \Subset I $ ($ a <b$), we have
\[
0\le u(x, b)-u(x, a)
\le A_0
\left(\max_{\tau\in[a,b]}u(x,\tau)^\mu\right)(b-a)
\]
and
\[
u(x, b)^{1-\mu}
\le u(x, a)^{1-\mu}+A_0(1-\mu)(b-a).
\]
Moreover, if $\phi\in C^{2,1}$ touches $u$ from above or below
at an interior point $(x_0,t_0)$, then
\[
0\le\phi_t(x_0,t_0)\le A_0u(x_0,t_0)^\mu.
\]
\end{lemma}

This lemma yields elliptic inequalities for fixed-time
slices, which we use in the proof of Lemma~\ref{Se3:lem1}.

\begin{lemma}\label{Section 2:lemma3}
Let $ u \geq 0 $ be a continuous viscosity solution of
  \begin{equation*}
    |Du|^{p} F(D^{2}u) - u_{t} = \lambda_{0}(x,t)\, u^{\mu}\, \chi_{\{u>0\}}(x,t)
    \qquad \text{in } \quad  B_{1} \times (a,b),
\end{equation*}
where $ 0<\mathcal{M}_{1}\leq \lambda_{0}(x,t)\leq \mathcal{M}_{2}<\infty $, and assume that $ u $ satisfies \hyperref[AC]{\bf (AC)}. Fix $ t_{*} \in (a, b) $, and set $ \omega(x):= u(x, t_{*}) $, then $ \omega $ solves, in the viscosity sense,
\begin{equation*}
  0 \leq |D\omega|^{p} F(D^{2}\omega) \leq (\mathcal{M}_{2}+A_0)\omega^{\mu}    \quad    \text{in}   \quad   B_{1}.
\end{equation*}
\end{lemma}

\begin{remark}
For ease of exposition, we defer the proofs of
Lemmas~\ref{Section 2:lemma2} and~\ref{Section 2:lemma3}
to \hyperref[Appendix:B]{Appendix B}.
\end{remark}
}



We now establish a stability theorem tailored to our setting. Its detailed proof can be found in~\cite[Theorem 2.10]{LLY24}.

\begin{theorem}[{\it Stability}]
\label{Section 2: Thm2}
Suppose that $ F_{k}, F: \mathrm{Sym}(n) \rightarrow   \mathbb{R} $ satisfy \hyperref[F1]{\bf (F1)} for $ k=1,2, \cdots $, with $ f_{k}, f \in C^{0}(Q_{1}) \cap L^{\infty}(Q_{1}) $. Let $ u_{k} $ be a viscosity subsolution (resp. supersolution) of
\begin{equation*}
|Du_{k}|^{p} F_k(D^{2}u_{k}) - (u_{k})_{t}  =  f_{k}(x,t)   \quad  \text{in}   \quad  Q_{1},
\end{equation*}
with $ 0 \leq p< \infty$. Moreover, assume that as $ k  \rightarrow  \infty $, $ u_{k} $ and $ f_{k} $ converge locally uniformly to $ u $ and $  f  $, respectively, and the sequence of operators $ F_{k} $ converges locally uniformly to $ F $ . Then $ u$ is a viscosity subsolution (resp.supersolution) of
\begin{equation*}
|Du|^{p} F(D^{2}u) - u_{t}  =  f(x,t)   \quad  \text{in}   \quad  Q_{1/2}.
\end{equation*}
\end{theorem}

The following result is the ``Cutting lemma'' from \cite[Lemma~6]{CL13}, and it is concerned with the homogeneous degenerate elliptic problem:
\begin{lemma}[{\it Cutting lemma}]
\label{Section 2: lemma2}
Let $ F $ be an operator satisfying \hyperref[F1]{\bf (F1)} and $ u$ be a viscosity solution of
\begin{equation*}
 |Du|^{p} F(D^{2}u) = 0   \quad   \text{in}  \quad  B_{1} \subset \mathbb{R}^{n},
\end{equation*}
with $ p \geq 0 $. Then $ u $ is a viscosity solution of
\begin{equation*}
 F(D^{2}u) = 0   \quad   \text{in}  \quad  B_{1}.
\end{equation*}
\end{lemma}

{
We will use \cite[Theorem 1.1]{CL16} which is stated below in a simplified form.

\begin{proposition}[{\it Interior H\"older estimate}]
\label{prop:IS-holder}
Let $u\in C^{0}(B_1)\cap L^\infty(B_1)$ be a viscosity solution to
\begin{equation*}
\mathscr{P}^{+}_{\lambda,\Lambda}
    (D^2u)\ge -\mathrm{C}_0   \quad   \text{in}   \quad  \{ |Du| > \gamma \}
\end{equation*}
and a viscosity solution to
\begin{equation*}
\mathscr{P}^{-}_{\lambda,\Lambda}
    (D^2u) \leq  \mathrm{C}_0   \quad   \text{in}   \quad  \{ |Du| > \gamma \},
\end{equation*}
for some fixed $ \gamma > 0 $, and $ \mathrm{C}_{0} \geq 0 $. Then there exists
$\beta=\beta(n,\lambda,\Lambda)\in(0,1)$ such that
$u\in C_{\mathrm{loc}}^{0,\beta}(B_1)$.
Moreover, for every $\tau\in(0,1)$,
\[
\|u\|_{C^{0,\beta}(B_\tau)}\le \mathrm{C},
\]
where $\mathrm{C} $ depends only on
$n,\lambda,\Lambda,\gamma,\mathrm{C}_0,
\|u\|_{L^\infty(B_1)}$, and $\tau$.
\end{proposition}
}

\subsection{The existence of viscosity solutions}
Finally, we shall establish the existence of a viscosity solution to the problem
\begin{equation}\tag{{\bf D-BVP}}
\label{Se2:D-BVP}
\left\{
     \begin{aligned}
     & |Du|^{p} F(D^{2}u) - u_{t} = \lambda_{0}(x,t) u^{\mu}\chi_{\{u>0\}}(x,t)   && \text{in}  \quad Q_{1},   \\
     &  u(x,t) = h(x,t)     &&   \text{on}  \quad  \partial B_{1} \times (-1,1),          \\
     & u(x,-1) = \widehat{u}(x)    &&   \text{in}  \quad \overline{B_{1}},
     \end{aligned}
     \right.
\end{equation}
for continuous, bounded nonnegative functions $ \widehat{u}$ and $ h $ fulfilling the compatibility condition $ h(x,-1) = \widehat{u}(x) $ for any $ x \in \partial B_{1} $. The existence of a nonnegative viscosity solution to \eqref{Se2:D-BVP} is established by employing the well-known Perron's method together with the preceding Lemma \ref{Se2:lem1}. Specifically,

\begin{theorem}
Suppose that $ F $ satisfies the assumption \hyperref[F1]{\bf (F1)}. In addition, $ h $ and $  \widehat{u} $ are continuous, bounded, nonnegative functions. If there exist a viscosity subsolution $ u_{\ast} $ to \eqref{Se2:D-BVP} and a viscosity supersolution $ u^{\ast} $ to \eqref{Se2:D-BVP} such that
\begin{equation*}
u^{\ast}   =    u_{\ast} \quad \text{on}  \quad \partial_{\mathrm{par}} Q_{1},
\end{equation*}
then there exists a nonnegative viscosity solution $ u \in C^{0}(\overline{B_{1}} \times [-1, 1)) $ to~\eqref{Se2:D-BVP}.
\end{theorem}

\vspace{3mm}
\section{Nondegeneracy of solutions and Improved regularity estimates}\label{Section 3}  
{{We prove Theorems~\ref{Thm1} and~\ref{Thm2}, starting
with the non-degeneracy estimate in Theorem~\ref{Thm2},
which supplies the universal constant $C_0^*$ needed
for the subsequent dyadic analysis.
We then establish a normalized decay lemma and a forward
comparison estimate. Together with dyadic iteration,
these results yield the sharp upper growth estimate
for normalized viscosity solutions in $Q_1$
(Theorem~\ref{Sec3:thm1}).
Finally, normalization and rescaling establish
Theorem~\ref{Thm1} for bounded nonnegative viscosity
solutions satisfying \hyperref[AC]{\textbf{(AC)}}.}}

\subsection{A barrier for the non-degeneracy of solutions} We now prove Theorem~\ref{Thm2}.

\begin{proof}[{\bf Proof of Theorem \ref{Thm2}}]
By the continuity of viscosity solutions, it suffices to prove the result for points
$ (x_{0}, t_{0}) \in \{u>0\} \cap Q_{T}$. We consider the comparison function
\begin{equation*}
  \widetilde{\Phi}(x,t):=
  \mathcal{C}_{1} \!\left(
  |x-x_{0}|^{\frac{2+p}{1+p}}+
  \mathcal{C}_{2}(t_{0}-t)^{\frac{1+p-\mu}{(1+p)(1-\mu)}}
  \right)^{\!\frac{1+p}{1+p-\mu}},
  \quad (x,t) \in Q_{r}^{-}(x_{0}, t_{0}),
\end{equation*}
where $ \mathcal{C}_{1}, \mathcal{C}_{2} > 0 $ satisfy
\begin{equation}\label{Section3:eq110}
  \mathcal{C}_{2} \leq
  \left(\frac{\mathcal{M}_{1}\mathcal{C}_{1}^{\mu-1}(1-\mu)}{2}\right)^{\!\frac{1+p-\mu}{(1+p)(1-\mu)}},
  \qquad
  \mathcal{C}_{1}:=
  \left(\frac{\mathcal{M}_{1}(1+p-\mu)^{2+p}}
  {2\Lambda(2+p)^{1+p}\{(n-2)(1+p-\mu)+2+p\}}\right)^{\!\frac{1}{1+p-\mu}}.
\end{equation}

For simplicity, we denote $ E:= |x-x_{0}|^{\beta}+
  \mathcal{C}_{2}(t_{0}-t)^{\delta} $, $ \beta:= \frac{2+p}{1+p} $, $ \delta:= \frac{1+p-\mu}{(1+p)(1-\mu)}$ and $ \gamma:= \frac{1+p}{1+p-\mu} $, then direct calculations yield that
\begin{equation*}
  |D \widetilde{\Phi}|^{p} = (\mathcal{C}_{1} \gamma  \beta E^{\gamma-1})^{p} |x-x_{0}|^{(\beta-1)p},
\end{equation*}
and
\begin{equation*}
 D^{2} \widetilde{\Phi} =  \mathcal{C}_{1} \gamma  \beta E^{\gamma-1} |x-x_{0}|^{\beta-2} \left[ \left( \frac{(\gamma-1)  \beta |x-x_{0}|^\beta}{E} + \beta - 2 \right) \frac{(x-x_{0})\otimes (x-x_{0})}{|x-x_{0}|^2} +  \bf{\mathrm{I}}d_{n} \right],
\end{equation*}
then by using the uniform ellipticity condition and $ F(\mathrm{O}_n) = 0 $ (\hyperref[F1]{\bf (F1)}), we obtain
\begin{align}\label{Section3:eq111}
\begin{split}
|D \widetilde{\Phi}|^{p} F(D^{2} \widetilde{\Phi})  & \leq  \Lambda (\mathcal{C}_{1} \gamma  \beta E^{\gamma-1})^{1+p}|x-x_0|^{(\beta-1)p+\beta-2} \bigg[ \frac{(\gamma-1)  \beta |x-x_{0}|^\beta}{E} + \beta + n-2 \bigg]      \\
& \leq \Lambda \mathcal{C}_{1}^{1+p} \left(\frac{2+p}{1+p-\mu}\right)^{1+p}  E^{\mu \gamma}\left(n-2+\frac{2+p}{1+p-\mu}\right),
\end{split}
\end{align}
where we have used
\begin{equation*}
(\beta-1)p+\beta-2 = 0,   \quad   (\gamma-1)(1+p) = \mu \gamma,  \quad  \beta \gamma = \frac{2+p}{1+p-\mu},    \quad       \text{and}   \quad   |x-x_{0}|^{\beta} \leq E
\end{equation*}
in the second inequality of \eqref{Section3:eq111}. Moreover, noticing that $ t_{0} -t \leq \left(   \frac{E}{\mathcal{C}_{2}} \right)^{\frac{1}{\delta}}   $, $ \delta \geq 1 $ and $ \gamma- \frac{1}{\delta} = \mu \gamma$, then a routine calculation reveals that
\begin{equation}\label{Section3:eq112}
- \widetilde{\Phi}_{t} = \mathcal{C}_{1} \mathcal{C}_{2} \gamma \delta E^{\gamma-1} (t_{0}-t)^{\delta-1} \leq     \mathcal{C}_{1}  \mathcal{C}_{2}^{\frac{1}{\delta}} \gamma \delta E^{\mu \gamma}.
\end{equation}
Now we combine \eqref{Section3:eq110}--\eqref{Section3:eq112} and $ 0 < \mathcal{M}_{1} \leq \lambda_{0}(x,t) $, we deduce that
\begin{equation*}
  |D\widetilde{\Phi}|^{p} F(D^{2}\widetilde{\Phi})
  - \widetilde{\Phi}_{t}
  - \lambda_{0}(x,t)\widetilde{\Phi}_{+}^{\mu} \leq 0
  \quad \text{in } \quad Q_{r}^{-}(x_{0}, t_{0}).
\end{equation*}
The preceding spatial computations are valid for $x\neq x_0$.
We now verify the supersolution inequality at $x=x_0$.
If $p=0$, the barrier is smooth for $t<t_0$, and the
inequality extends to $x=x_0$ by continuity.
If $p>0$, fix $t\in(t_0-r^\theta,t_0)$ and let
$\phi\in C^{2,1}$ touch $\widetilde\Phi$ from below at $(x_0,t)$.
Since $\widetilde\Phi$ is differentiable at this point, we have
\[
D\phi(x_0,t)=D\widetilde\Phi(x_0,t)=0,
\quad
\phi_t(x_0,t)=\widetilde\Phi_t(x_0,t).
\]
By~\eqref{Section3:eq112} and the choice of $\mathcal{C}_2$ in~\eqref{Section3:eq110},
we obtain
\[
\begin{aligned}
&|D\phi(x_0,t)|^p F(D^2\phi(x_0,t))
-\phi_t(x_0,t)-\lambda_0(x_0,t)\widetilde\Phi_{+}(x_0,t)^\mu\\
&\qquad =
-\widetilde\Phi_t(x_0,t)
-\lambda_0(x_0,t)\widetilde\Phi_{+}(x_0,t)^\mu
\le -\frac{\mathcal M_1}{2}\widetilde\Phi_{+}(x_0,t)^\mu
\le 0.
\end{aligned}
\]
Thus $\widetilde\Phi$ is a viscosity supersolution in
$B_r(x_0)\times(t_0-r^\theta,t_0)$. The comparison below extends to the terminal time $t_0$ by continuity.

We claim that there exists a point
$ (x', t') \in \partial_{\mathrm{par}}Q_{r}^{-}(x_{0}, t_{0}) $
such that $ u(x', t') \geq  \widetilde{\Phi}(x',t') $.
Indeed, if $ \widetilde{\Phi} \geq  u $ on
$ \partial_{\mathrm{par}} Q_{r}^{-}(x_{0}, t_{0}) $,
then by the comparison principle (Lemma \ref{Se2:lem1}) we would have
$ u \leq \widetilde{\Phi} $ in $ Q_{r}^{-}(x_{0}, t_{0}) $,
which contradicts $ \widetilde{\Phi}(x_{0}, t_{0})  = 0 < u(x_{0}, t_{0}) $.
Hence, the claim follows.
Since $ \widetilde{\Phi}(x',t') \geq  \mathcal{C}_{1} \min \{ 1, \mathcal{C}_{2}\}^{\gamma} r^{\frac{2+p}{1+p-\mu}} $, the desired estimate is obtained.
\end{proof}

Before stating the proof of Theorem~\ref{Thm1}, we make the following definition.

\begin{definition}\label{Sec3:def1}
  For any fully nonlinear operator $ F $ satisfying \hyperref[F1]{\bf (F1)}, we say that $ u \in \mathbb{J}(F, \lambda_{0}, \mu, A_0)(Q_{1})$ if

\begin{enumerate}[I)]

\item  \label{b1} $ |Du|^{p} F(D^{2}u) - u_{t} = \lambda_{0}(x,t) u^{\mu}\chi_{\{u>0\}}(x,t) $ in $ Q_{1} $ with $ 0 <  \mu < \min\{1, p+1\}=1  $;

\item \label{b2} $ 0 \leq u \leq 1 $, and $ 0< \mathcal{M}_{1} \leq \lambda_{0}(x,t) \leq \mathcal{M}_{2} <\infty $ in $ Q_{1} $;

\item  \label{bbb}   Condition \hyperref[AC]{\bf (AC)} holds with the same constant $ A_{0} $;

\item $ u(0,0) = 0 $.
\end{enumerate}
\end{definition}

{{\begin{remark}
Assumption~\ref{bbb}, namely
\hyperref[AC]{\textbf{(AC)}} in
Definition~\ref{Sec3:def1}, serves a technical role.
The inequality $u_t\ge0$ ensures that the supremum
over a backward cylinder equals the supremum on its
terminal time slice.
Together with the upper bound $u_t\le A_0u^\mu$,
it yields the temporal estimates and elliptic
inequalities for fixed-time slices in
Lemmas~\ref{Section 2:lemma2} and~\ref{Section 2:lemma3}.
Combined with Proposition~\ref{prop:IS-holder},
these estimates yield compactness and ensure that
the limit is stationary; see Lemma~\ref{Se3:lem1}.
\end{remark}
}}

\begin{remark}    
The condition $\partial_t u\ge0$ is also natural from
a macroscopic perspective.
More precisely, let $u$ be a nonnegative viscosity solution
of the initial--boundary value problem associated
with~\eqref{DCP}, with boundary datum $\varphi$:
\begin{equation}\tag{{\bf D-BVP1}}
\label{Se3:D-BVP1}
\left\{
\begin{aligned}
& |Du|^pF(D^2u)-\partial_tu
  =\lambda_0(x,t)u^\mu\chi_{\{u>0\}}(x,t)
&& \text{in }\Omega_T,\\
& u(x,t)=\varphi(x,t)
&& \text{on }\partial_{\mathrm{par}}\Omega_T,
\end{aligned}
\right.
\end{equation}
where $\lambda_0$ is nonincreasing in $t$,
$\varphi$ is nondecreasing in $t$, and $\varphi(x,0)=0$.
We derive the time monotonicity by a comparison argument
(cf.~\cite[Appendix]{Sh03}).

Fix $0<h<T$ and define
\[
u^h(x,t):=u(x,t+h),
\quad (x,t)\in\Omega\times(0,T-h).
\]
Since $\lambda_0(x,t+h)\le\lambda_0(x,t)$,
the function $u^h$ satisfies
\[
|Du^h|^pF(D^2u^h)-(u^h)_t
-\lambda_0(x,t)(u^h)_+^\mu\le0
\quad\text{in }\Omega\times(0,T-h)
\]
in the viscosity sense.
Thus $u^h$ is a viscosity supersolution of the
equation in~\eqref{Se3:D-BVP1}.

On the lateral boundary, the monotonicity of $\varphi$
gives
\[
u(x,t)=\varphi(x,t)
\le\varphi(x,t+h)=u^h(x,t).
\]
On the initial face, the non-negativity of $u$ yields
\[
u(x,0)=0\le u(x,h)=u^h(x,0).
\]
Applying Lemma~\ref{Se2:lem1}, we obtain
\[
u(x,t)\le u(x,t+h)
\quad\text{in }\Omega\times(0,T-h).
\]
Since $h$ is arbitrary, $u$ is non-decreasing in time.
Consequently, $\partial_tu\ge0$ in the viscosity sense.
\end{remark}

Next, we shall define $\displaystyle L_{(r,x_{0},t_{0})}[u] := \sup_{Q_{r}^{-}(x_{0},t_{0})} u(x,t)$ and the set
\begin{equation*}
  \mathbb{V}_{p,\mu}[u]:= \bigg\{j \in \mathbb{N} \cup \{0\}:  L_{1/2^{j}}[u] \leq 2^{\frac{2+p}{1+p-\mu}} \max \bigg\{1, \frac{1}{\mathrm{C}_{0}^{*}} \bigg\} L_{1/2^{j+1}}[u] \bigg\},
\end{equation*}
where $ \mathrm{C}_{0}^{*} = \mathrm{C}_{0}^{*}(n, \lambda, \Lambda, \mathcal{M}_{1}, \mathcal{M}_{2}, p, \mu) $ is the positive constant in Theorem \ref{Thm2}. Noticing that the set $ \mathbb{V}_{p,\mu}[u] \neq \emptyset $. In fact, we distinguish two cases.

\medskip
\noindent
{\bf Case 1.} If $(0,0)\in\partial\{u>0\}$, then Theorem~\ref{Thm2} and the continuity of $u$ give
\[
L_{1/2}[u]\ge \mathrm{C}_0^*2^{-\frac{2+p}{1+p-\mu}}.
\]
Since $0\le u\le1$, we obtain
\[
L_1[u]\le1
\le\max\{1,\mathrm{C}_0^*\}
\le2^\frac{2+p}{1+p-\mu}\max\{1,1/\mathrm{C}_0^*\}L_{1/2}[u].
\]
Hence $0\in\mathbb{V}_{p,\mu}[u]$.

\medskip
\noindent
{\bf Case 2.} If $(0,0)\notin\partial\{u>0\}$, then the continuity
and nonnegativity of $u$, together with $u(0,0)=0$,
imply that $u$ vanishes in a neighborhood of $(0,0)$.
Thus there exists an integer $j_0\ge0$ such that
$u\equiv0$ in $Q_{2^{-j_0}}$. Consequently,
\[
L_{2^{-j}}[u]=L_{2^{-(j+1)}}[u]=0
\quad\text{for every }j\ge j_0.
\]
By the definition of $\mathbb{V}_{p,\mu}[u]$,
every $j\ge j_0$ belongs to this set. Therefore, $\mathbb{V}_{p,\mu}[u]\neq  \emptyset $ in both cases.


\subsection{Normalized decay lemma}
We begin by deriving improved regularity estimates for functions on $\mathbb{J}(F, \lambda_{0}, \mu, A_0)(Q_{1})$ along the free boundary.

\begin{lemma}[{\it Normalized decay lemma}]
\label{Se3:lem1}
There exists a positive constant $ \mathrm{C}_{0}  =\mathrm{C}_0(n,\lambda,\Lambda,p,\mu,
\mathcal{M}_1,\mathcal{M}_2,A_0) $ such that
\begin{equation}\label{Sec3:eq1}
  L_{1/2^{j+1}}[u] \leq \mathrm{C}_{0} \bigg(  \frac{1}{2^{j}} \bigg)^{\frac{2+p}{1+p-\mu}}
\end{equation}
for all $ u \in \mathbb{J}(F, \lambda_{0}, \mu, A_0)(Q_{1}) $ and $ j \in  \mathbb{V}_{p,\mu}[u] $.
\end{lemma}
{
\begin{proof}
We divide the proof into the following four steps.

\medskip
\noindent
{\bf Step 1.} We argue by contradiction. Suppose there exist $u_{k} \in \mathbb{J}(F, \lambda_{0}, \mu, A_0)(Q_{1})$ and $j_{k} \in \mathbb{V}_{p,\mu}[u_{k}]$ such that
\begin{equation}\label{Sec3:eq2}
  L_{1/2^{j_{k}+1}}[u_{k}] > k \left( \frac{1}{2^{j_{k}}} \right)^{\frac{2+p}{1+p-\mu}}.
\end{equation}
Define
\begin{equation*}
  \alpha_{k} := \frac{1}{2^{(p+2)j_{k}}} \frac{1}{(L_{1/2^{j_{k}+1}}[u_{k}])^{p}},
  \qquad \text{and} \qquad
  v_{k}(x,t) := \frac{u_{k}(\frac{1}{2^{j_{k}}}x, \alpha_{k}t)}{L_{1/2^{j_{k}+1}}[u_{k}]}
  \quad \text{in } \quad Q_{1}.
\end{equation*}
To verify that $v_k$ is well-defined in $Q_1^-$, we track the spatial and temporal scales. For any $(x,t)\in Q_1^-$, the rescaled spatial variable satisfies $|2^{-j_k}x|\leq 2^{-j_k}$. For the temporal variable, we have $\alpha_k t\in[-\alpha_k,0]$. By \eqref{Sec3:eq2}, one has \begin{equation*} \alpha_k \leq \frac{2^{-(p+2)j_k}}{(k\cdot 2^{-j_k\alpha})^p} = k^{-p}2^{-j_k(p+2-\alpha p)}, \quad \text{for} \quad \alpha=\frac{2+p}{1+p-\mu}. \end{equation*} A direct computation yields $ p+2-\alpha p = \frac{(2+p)(1-\mu)}{1+p-\mu} = \theta$. Hence, \[ \alpha_k \leq k^{-p}2^{-j_k\theta}\leq 2^{-j_k\theta} \quad   \text{for all }k\geq1. \] Therefore, the arguments of $u_k$ remain in the intrinsic cylinder $Q_{2^{-j_k}}$, and $v_k$ is well-defined on $Q_1^-$ for every $ k $.

\medskip
\noindent
{\bf Step 2.} A direct computation shows that $v_{k}$ satisfies:
\begin{itemize}
\item[(i)] $0 \leq v_{k}(x,t) \leq \dfrac{L_{1/2^{j_{k}}}[u_{k}]}{L_{1/2^{j_{k}+1}}[u_{k}]} \leq \Theta := 2^{\frac{2+p}{1+p-\mu}} \max \big\{1, \tfrac{1}{\mathrm{C}_{0}^{*}}\big\}$ in $Q_{1}^{-}$, $v_{k}(0,0) = 0$, and $\displaystyle \sup_{B_{1/2}} v_{k}(x,0) =1 $;

\item[(ii)] $L_{1/2}[v_{k}] \geq 1$;

\item[(iii)]  \label{No1} $ \partial_{t} v_{k} $ still satisfies \hyperref[AC]{\bf (AC)}, namely, $ 0 \leq \partial_{t} v_{k} \leq A_{0} \epsilon_{k} v_{k}^{\mu} $ in $Q_{1}$ (in the viscosity sense), where
    \begin{equation*}
      \epsilon_{k}:= \frac{1}{2^{(p+2)j_{k}}} \frac{1}{(L_{1/2^{j_{k}+1}}[u_{k}])^{1+p-\mu}}  \overset{\eqref{Sec3:eq2}}{\leq} \frac{1}{k^{1+p-\mu}}  \to 0  \quad  \text{as}  \quad  k  \to \infty.
    \end{equation*}

\item[(iv)] \label{No2} In the viscosity sense,
\begin{equation*}
  |Dv_{k}|^{p} \widetilde{F}_{k}(D^{2}v_{k}) - (v_{k})_{t}
  = \frac{1}{2^{(p+2)j_{k}}} \frac{1}{(L_{1/2^{j_{k}+1}}[u_{k}])^{1+p-\mu}}
  \lambda_{0}\!\left(\tfrac{1}{2^{j_{k}}}x, \alpha_{k}t\right) (v_{k})_{+}^{\mu} \quad \text{in } \quad Q_{1},
\end{equation*}
where
\begin{equation*}
  \widetilde{F}_{k}(\mathrm{X})
  := \frac{1}{2^{2j_{k}}L_{1/2^{j_{k}+1}}[u_{k}]}
  F\big(2^{2j_{k}}L_{1/2^{j_{k}+1}}[u_{k}]\mathrm{X}\big).
\end{equation*}
\end{itemize}
It is straightforward to verify that $\widetilde{F}_k$ still satisfies \hyperref[F1]{\bf (F1)}. Applying Lemma~\ref{Section 2:lemma2} to $ v_{k} $, then we have
\begin{equation}\label{Sec3:eq299}
  |v_{k}(x,t) - v_{k}(x, t')| \leq A_{0}\Theta^{\mu} \epsilon_{k} |t-t'|, \quad  -1 < t, t' <0.
\end{equation}

\medskip
\noindent
{\bf Step 3.} Set $ \omega_{k}(x):= v_{k}(x, 0) $, then Lemma \ref{Section 2:lemma3}, applied (iii) and (iv), gives
\begin{equation}\label{Sec3:eq300}
  0 \leq |D\omega_{k}|^{p} \widetilde{F}_{k}(D^{2}\omega_{k}) \leq h_{k} \quad   \text{in}   \quad  B_{1},
\end{equation}
where $ h_{k}:= (\mathcal{M}_{2}+A_{0}) \Theta^{\mu} \epsilon_{k} \to 0 $, as $ k \to \infty $. Next, we want to prove the following claim:

{\it Claim}. $ \omega_{k} \in C^{0}(B_1)\cap L^\infty(B_1) $ is a viscosity solution
\begin{equation}\label{Sec3:eq301}
\mathscr{P}^{+}_{\lambda,\Lambda}
    (D^2 \omega_{k})\ge -\mathrm{C}_0   \quad   \text{in}   \quad  \{ |D\omega_{k}| > 1\}
\end{equation}
and is a viscosity solution to
 \begin{equation}\label{Sec3:eq302}
\mathscr{P}^{-}_{\lambda,\Lambda}
    (D^2 \omega_{k})\leq  \mathrm{C}_0   \quad   \text{in}   \quad  \{ |D\omega_{k}| > 1\},
\end{equation}
where
\begin{equation*}
  \mathrm{C}_{0}:=\max\{1,\Theta, (\mathcal{M}_{2}+A_{0}) \Theta^{\mu}\}.
\end{equation*}
Indeed, we only verify~\eqref{Sec3:eq301}, and the remaining proof is similar. Let $ \psi \in C^{2} $ touch $ \omega_k $ from above at $ x_0 \in B_{1} $, and $ |D \psi(x_0 )| > 1 $, by the left-hand inequality of~\eqref{Sec3:eq300}, we obtain
\begin{equation*}
  |D \psi(x_0)|^{p} \widetilde{F}_{k}(D^{2}\psi(x_0))  \geq 0,
\end{equation*}
which implies that $ \widetilde{F}_{k}(D^{2}\psi(x_0))  \geq 0 $ since $ |D \psi(x_0)|^{p} \geq 0 $. By uniform ellipticity of $ \widetilde{F}_{k}$,  and by $\widetilde{F}_{k}(\mathrm{O}_n) = 0 $ (\hyperref[F1]{\bf (F1)}), we have
\begin{equation*}
  \mathscr{P}^{+}_{\lambda,\Lambda}
    (D^2 \psi(x_0)) \ge 0  \geq -\mathrm{C}_0,
\end{equation*}
which shows \eqref{Sec3:eq301}. This proves the {\it Claim}.

Now we apply Proposition \ref{prop:IS-holder} to $ \omega_{k} $, then $ \omega_{k} \in C_{\mathrm{loc}}^{0,\beta}(B_1) $ for some $ \beta = \beta(n,\lambda,\Lambda) \in (0,1) $, then up to a subsequence, $\{\omega_{k}\}_{k\geq1}$ converges locally uniformly in $ B_1 $ to a continuous function $ \omega_{0} $. Moreover, by \eqref{Sec3:eq299}, we find that $ \{v_k\}_{k\geq1} $ converges locally uniformly to the stationary function $ \omega_{0} $.

\medskip
\noindent
{\bf Step 4.} By uniform ellipticity and $\widetilde{F}_{k}(\mathrm{O}_n)=0$(\hyperref[F1]{\bf (F1)}), after passing to a further subsequence, we have $\widetilde{F}_{k} \to \widetilde{F}_{0} $ locally uniformly in $ \mathrm{Sym}(n) $. Hence, using the stability of viscosity solutions (see \cite[Remarks~6.3]{CIL92}) and passing to the limit in \eqref{Sec3:eq300}, it follows that $ \omega_{0} $ solves, in the viscosity sense,
\begin{equation}\label{Sec3:eq303}
  |D\omega_{0}|^{p} \widetilde{F}_{0}(D^{2}\omega_{0}) = 0  \quad   \text{in}  \quad B_1.
\end{equation}
If $p=0$, the result is immediate; see \cite[Lemma~3.2]{SO19} for details. If $p>0$, we apply the cutting lemma (Lemma~\ref{Section 2: lemma2}) to \eqref{Sec3:eq303}, which yields $\widetilde{F}_0(D^2\omega_0)=0 $ in $ B_{1/2}$. The strong minimum principle (\cite[Theorem~2.1]{DaLio04}) then yields $\omega_{0} \equiv 0$, contradicting $\displaystyle \sup_{B_{1/2}} \omega_{0}(x) = 1$.
\end{proof}
}

{A comparison argument extends the spatial growth estimate at time $t_\ast$ to the corresponding forward parabolic cylinders.

\begin{proposition}\label{Section3:prop31}
Assume that $ F $ satisfies \hyperref[F1]{\bf (F1)}, $ 0<\mathcal{M}_{1}\leq \lambda_{0}(x,t)\leq \mathcal{M}_{2}<\infty$ and the condition~\hyperref[AC]{\bf (AC)} holds. Let $ u $ be a nonnegative viscosity solution of
  \begin{equation}\label{Sec3:prop31.eq1}
     |Du|^{p} F(D^{2}u) - u_{t} = \lambda_{0}(x,t) u^{\mu}\chi_{\{u>0\}}(x,t) \quad   \text{in}  \quad  B_{R}(x_{*}) \times  [t_{\ast}, t_{\ast}+ R^{\theta}).
  \end{equation}
If for some $ K \geq 0 $,
\begin{equation}\label{Sec3:prop31.eq2}
  u(x, t_{\ast}) \leq K|x-x_{\ast}|^{\frac{p+2}{1+p-\mu}}, \quad  x \in B_{R}(x_{\ast}),
\end{equation}
then for every $ 0 < r < R $,
\begin{equation*}
  \sup_{B_{r}(x_{\ast})\times [t_{\ast}, t_{\ast}+r^{\theta})} u  \leq \mathrm{C} r^{\frac{p+2}{1+p-\mu}},
\end{equation*}
where $ \mathrm{C} $ is a positive constant depending only on $ n, p, \mu, A_{0} $, $ K, \Lambda $ and $ \mathcal{M}_{1} $.
\end{proposition}

\begin{proof}
After a translation, we may assume $ (x_{\ast}, t_{\ast}) = (0, 0) $. We divide the proof into three steps.

\medskip
\noindent
{\bf Step 1.} We construct the barrier function
\begin{equation*}
  \Phi_{r}(x,t)
  = \mathcal{C} \bigg(|x|^{\frac{2+p}{1+p}}
  + \mathcal{B} (t+r^{\theta})^{\frac{1+p-\mu}{(1+p)(1-\mu)}}\bigg)^{\!\frac{1+p}{1+p-\mu}},
\end{equation*}
where $\mathcal{C}$ is the constant defined by
\begin{equation*}
  \mathcal{C}
  := \bigg[ \bigg(\frac{1+p-\mu}{2+p}\bigg)^{\!1+p}
  \frac{\mathcal{M}_{1}}{\Lambda}
  \frac{1}{n-2+\frac{2+p}{1+p-\mu}} \bigg]^{\!\frac{1}{1+p-\mu}}.
\end{equation*}
Here $ \mathcal{B}  > 0 $ will be chosen \textit{a posteriori}. We first show that, for every $ \mathcal{B} > 0$, then $ \Phi_{r} $ satisfies
\begin{equation}\label{Sec3:prop31.eq3}
  |D\Phi_{r}|^{p} F(D^{2}\Phi_{r}) - (\Phi_{r})_{t}
  - \lambda_{0}(x,t)(\Phi_{r})_{+}^{\mu} \leq 0
  = |Du|^{p} F(D^{2}u) - u_{t}
  - \lambda_{0}(x,t)u_{+}^{\mu}  \quad   \text{in}  \quad B_{r} \times [0, r^{\theta}).
\end{equation}
To facilitate the calculation, let $ E(x,t)= |x|^{\beta} + \mathcal{B}(t+r^{\theta})^{\delta} $ with $ \beta= \frac{2+p}{1+p} $, $ \delta = \frac{1+p-\mu}{(1+p)(1-\mu)}$, $ \gamma= \frac{1+p}{1+p-\mu}$, and $ \alpha:=\frac{2+p}{1+p-\mu}$, we compute
\begin{equation*}
    D_i \Phi_{r} = \mathcal{C} \gamma E^{\gamma-1}  \beta |x|^{\beta-2} x_i,
\end{equation*}
and
\begin{align*}
D_{ij}\Phi_{r} &  = \mathcal{C} \gamma  \beta E^{\gamma-1} \bigg[|x|^{\beta-2} \delta_{ij} + (\beta-2) |x|^{\beta-4} x_{i} x_{j} \bigg]  +  \mathcal{C} \gamma (\gamma-1)  \beta^{2} E^{\gamma-2} |x|^{2\beta-4}x_{i} x_{j}        \\
& = \mathcal{C} \gamma  \beta E^{\gamma-1} |x|^{\beta-2} \left[ \left( \frac{(\gamma-1)  \beta |x|^\beta}{E} + \beta - 2 \right) \frac{x_i x_j}{|x|^2} + \delta_{ij} \right], \quad  i,j=1, \cdots, n.
\end{align*}
This shows that the Hessian matrix $ D^{2}\Phi_{r} $ possesses an eigenvalue of multiplicity $ 1 $ given by
\begin{equation*}
\lambda_{1}= \mathcal{C} \gamma  \beta E^{\gamma-1} |x|^{\beta-2} \left( \frac{(\gamma-1)  \beta |x|^\beta}{E} + \beta - 1 \right),
\end{equation*}
and an eigenvalue of multiplicity $ n-1 $ given by
\begin{equation*}
  \lambda_{2}, \dots, \lambda_{n}= \mathcal{C} \gamma  \beta E^{\gamma-1} |x|^{\beta-2}.
\end{equation*}
We note that $ \lambda_{1}, \dots, \lambda_{n} > 0 $ since $ \beta, \gamma > 1 $, then by using the uniform ellipticity condition and $ F(\mathrm{O}_n) = 0 $ (\hyperref[F1]{\bf (F1)}), we obtain
\begin{equation}\label{Sec3:prop31.eq4}
F(D^{2}\Phi_{r}) \leq  \mathscr{P}^{+}_{\lambda,\Lambda}(D^{2}\Phi_{r})\leq \Lambda \mathcal{C} \gamma  \beta E^{\gamma-1} |x|^{\beta-2} \bigg[ \frac{(\gamma-1)  \beta |x|^\beta}{E} + \beta + n-2 \bigg].
\end{equation}
Moreover, a direct computation gives
\begin{equation*}
|D\Phi_{r}|^{p} =  (\mathcal{C} \gamma  \beta E^{\gamma-1})^{p} |x|^{(\beta-1)p},
\end{equation*}
which, together with \eqref{Sec3:prop31.eq4}, yields 
\begin{align*}
|D\Phi_{r}|^{p} F(D^{2} \Phi_{r})  & \leq  \Lambda (\mathcal{C} \gamma \beta E^{\gamma-1})^{1+p}|x|^{(\beta-1)p+\beta-2} \bigg[ \frac{(\gamma-1)  \beta |x|^\beta}{E} + \beta + n-2 \bigg]   \\
& \leq \Lambda (\mathcal{C} \gamma  \beta) ^{1+p} E^{\mu \gamma}\left(n-2+\frac{2+p}{1+p-\mu}\right),
\end{align*}
where we used the identities:\begin{equation*}
(\gamma-1)(1+p) = \mu \gamma, \quad   \text{and}   \quad  (\beta-1)p+\beta-2 = 0.
\end{equation*}

Moreover, it is obvious that $ (\Phi_{r})_{t} \geq 0 $, then combining the value of $ \mathcal{C} $, it follows
\begin{equation*}
|D\Phi_{r}|^{p} F(D^{2}\Phi_{r}) - (\Phi_{r})_{t} \leq \mathcal{M}_{1} (\Phi_{r})^{\mu} \leq  \lambda_{0}(x,t) (\Phi_{r})^{\mu},
\end{equation*}
which implies that \eqref{Sec3:prop31.eq3} holds. At $x=0$, the supersolution inequality follows from a viscosity argument analogous to that used in the proof of Theorem~\ref{Thm2}.

\medskip
\noindent
{\bf Step 2.} On the initial time slice $ t =0 $, \eqref{Sec3:prop31.eq2} gives $ u(x,0) \leq K r^{\frac{2+p}{1+p-\mu}} $. On the lateral boundary $ |x| =r $, we combine Lemma \ref{Section 2:lemma2} and \eqref{Sec3:prop31.eq2} to obtain
\begin{align*}
u(x,t)^{1-\mu}  & \leq u(x,0)^{1-\mu}  + A_{0}(1-\mu)t    \\
 & \leq  [K^{1-\mu} + A_{0}(1-\mu)] r^{\theta}, \quad  0 \leq t \leq r^{\theta},
\end{align*}
where we have used the equality $ \alpha(1-\mu) = \theta $. Now, we set
\begin{equation*}
  M:= [K^{1-\mu} + A_{0}(1-\mu)]^{\frac{1}{1-\mu}}    \quad    \text{and}  \quad   L:= \max \{1, M\},
\end{equation*}
then it is easy to see that $ u \leq L r^{\alpha} $ on $ \partial_{\mathrm{par}}(B_{r} \times [0, r^{\theta})) $. With $ \mathcal{C} $ now fixed, choose $ \mathcal{B} = \left( \frac{L}{\mathcal{C}}  \right)^{\frac{1}{\gamma}}  $, so that $ \mathcal{C} \mathcal{B}^{\gamma} = L $. Since $ t + r^{\theta} \geq r^{\theta} $, then on the parabolic boundary, we obtain
\begin{equation*}
  \Phi_{r}(x,t) \geq \mathcal{C} \mathcal{B}^{\gamma} r^{\theta \delta \gamma} = Lr^{\alpha},
\end{equation*}
where we used $ \theta \delta \gamma= \beta \gamma = \alpha $, hence $ u \leq  \Phi_{r} $ on $ \partial_{\mathrm{par}}(B_{r} \times [0, r^{\theta}))  $.

\medskip
\noindent
{\bf Step 3.} By the comparison principle (Lemma~\ref{Se2:lem1}), we conclude that $\Phi_{r} \geq u $ in $ B_{r} \times [0, r^{\theta}) $. For $|x| \le r$ and $0 \le t < r^{\theta}$, the identity $\theta\delta = \beta$ gives
\[
  E(x,t) \le r^{\beta} + \mathcal{B}(2r^{\theta})^{\delta}
  = (1 + 2^{\delta}\mathcal{B})\,r^{\beta}.
\]
It follows that
\[
  \Phi_{r} \le \mathcal{C}(1 + 2^{\delta}\mathcal{B})^{\gamma} r^{\alpha}.
\]
To estimate this coefficient, we use $(a+d)^{\gamma} \le 2^{\gamma-1}(a^{\gamma} + d^{\gamma})$. Using $ \mathcal{C}  \mathcal{B}^{\gamma} = L$ and $\delta \gamma = \frac{1}{1-\mu}$, it follows that
\begin{align*}
  \mathcal{C}(1 + 2^{\delta}\mathcal{B})^{\gamma}
  &\le 2^{\gamma-1}\bigl[ \mathcal{C} + 2^{\delta \gamma}\mathcal{C}\mathcal{B}^{\gamma} \bigr] \\
  &\le 2^{\gamma-1}\bigl[ \mathcal{C} + 2^{\frac{1}{1-\mu}}L \bigr].
\end{align*}
Together with $u \le \Phi_{r}$, this proves the required bound.  
\end{proof}
}


{ Using Proposition~\ref{Section3:prop31}, we establish an improved regularity estimate for the class
$\mathbb{J}(F,\lambda_0,\mu,A_0)(Q_1)$. The proof draws on \cite[Theorem~2.2]{Sh03} but requires additional technical arguments.   

\begin{theorem}\label{Sec3:thm1}
  There exists a positive constant $ \mathrm{C} = \mathrm{C}(n, \lambda, \Lambda, p, \mu, \mathcal{M}_{1}, \mathcal{M}_{2}, A_0)$ such that, for every $ u \in \mathbb{J}(F, \lambda_{0}, \mu, A_0)(Q_{1}) $
  \begin{equation}\label{Section 3:Thm3.1:main}
    u(x,t) \leq \mathrm{C} \cdot \Gamma (x,t)^{\frac{2+p}{1+p-\mu}} \ \ \text{for all} \ \ (x,t) \in Q_{1/2},
  \end{equation}
where
 \begin{align*}
 \Gamma(x,t):= \left\{
     \begin{aligned}
      & \sup \{r\geq 0: Q_{r}(x,t) \subset \{u>0\}\}, && \text{for} \ \ (x,t) \in \{u>0\},        \\
     & 0,      &&       \text{otherwise}.
     \end{aligned}
     \right.
\end{align*}
\end{theorem}

\begin{proof}
The argument proceeds by induction. We first claim that
\begin{equation}\label{Se3:eq6}
   L_{1/2^{j+1}}[u] \leq \mathrm{C}_{0} \bigg( \frac{1}{2^{j}} \bigg)^{\!\frac{2+p}{1+p-\mu}}
   \quad \text{for all} \quad j \geq 0,
\end{equation}
where $\mathrm{C}_{0}$ is the constant given by Lemma~\ref{Se3:lem1}.
Without loss of generality, we assume $\mathrm{C}_{0} \geq 1$. The estimate in \eqref{Se3:eq6} clearly holds for $j = 0$.
Let $j\ge1$, and assume that \eqref{Se3:eq6} holds
with $j-1$ in place of $j$. We prove the estimate at index $j$ by distinguishing the following two cases.

\smallskip
\noindent
\textbf{Case 1.} If $j \in \mathbb{V}_{p,\mu}[u]$, the claim follows directly from Lemma~\ref{Se3:lem1};

\smallskip
\noindent
\textbf{Case 2.} If $j \notin \mathbb{V}_{p,\mu}[u]$, the definition of $ \mathbb{V}_{p,\mu}[u] $, and the induction hypothesis give
\begin{align*}
  L_{1/2^{j+1}}[u]
  &  < 2^{-\frac{2+p}{1+p-\mu}} \min\{1, \mathrm{C}_0^*\} L_{1/2^j}[u] \leq  2^{-\frac{2+p}{1+p-\mu}} \min\{1, \mathrm{C}_0^*\} \mathrm{C}_{0} \bigg( \frac{1}{2^{j-1}} \bigg)^{\!\frac{2+p}{1+p-\mu}}  \\
&  = \min\{1, \mathrm{C}_0^*\} \mathrm{C}_{0}
\bigg( \frac{1}{2^{j}} \bigg)^{\!\frac{2+p}{1+p-\mu}}  \leq  \mathrm{C}_{0}
\bigg( \frac{1}{2^{j}} \bigg)^{\!\frac{2+p}{1+p-\mu}}.
\end{align*}
Hence, \eqref{Se3:eq6} holds for all $j \in \mathbb{N} $.

\smallskip
Next, let $r \in (0,1)$ and choose $j \geq 0$ as the unique integer satisfying
$2^{-(j+1)} \leq r < 2^{-j}$. Then
\begin{equation}\label{Section3:eq666}
  L_{r}[u] \leq L_{1/2^{j}}[u]
  \leq \mathrm{C}_{0} \bigg( \frac{1}{2^{j-1}} \bigg)^{\!\frac{2+p}{1+p-\mu}}
  \leq  4^{\frac{2+p}{1+p-\mu}} \mathrm{C}_{0} \, r^{\frac{2+p}{1+p-\mu}}
  =: \mathrm{C}_{1}(n,\lambda,\Lambda,\mu,p,\mathcal{M}_{1},\mathcal{M}_2, A_0) \, r^{\frac{2+p}{1+p-\mu}}.
\end{equation}
We next obtain a uniform growth estimate around
nearby free-boundary points. Fix
\[
 z_*=(x_*,t_*)\in\partial\{u>0\}\cap Q_{3/4}.
\]
By continuity, $u(x_*,t_*)=0$.
Choose a fixed $0<\rho<1/4$, sufficiently small that
\[
 B_\rho(x_*)\times
 (t_*-\rho^{p+2},t_*+\rho^{p+2})
 \Subset Q_1
 \quad  \text{for every }z_*\in Q_{3/4}.
\]
Such a uniform choice is possible because
$ \overline{Q_{3/4}}\Subset Q_1 $. Now, define the rescaled function
\[
 \widehat u(x,t)
 :=u(x_*+\rho x,t_*+\rho^{p+2}t),
 \quad (x,t)\in Q_1.
\]
A direct computation shows that $ \widehat u $ satisfies, in the viscosity sense:
\[
 |D\widehat u|^p\widehat F(D^2\widehat u)
 -\widehat u_t
 =\widehat\lambda_0(x,t)\widehat u^\mu_{+}
 \quad \text{in } \quad Q_1,
\]
where
\[
 \widehat F(\mathrm{X}):=\rho^2F(\rho^{-2}\mathrm{X})
 \quad    \text{and}  \quad
 \widehat\lambda_0(x,t)
 :=\rho^{p+2}
   \lambda_0(x_*+\rho x,t_*+\rho^{p+2}t).
\]
Under this rescaling, condition \hyperref[AC]{\bf (AC)} becomes
\[
 0\leq\widehat u_t
 \leq\rho^{p+2}A_0\widehat u^\mu
 \quad   \text{in the viscosity sense}.
\]
The operator $\widehat F $ still satisfies \hyperref[F1]{\bf (F1)} with the same ellipticity constants. Moreover, 
\[
 0<\widehat\lambda_0\leq\mathcal M_2,
 \quad  0\leq\widehat u\leq1   \quad    \text{and}  \quad \widehat u(0,0)=0.
\]
Hence, applying \eqref{Section3:eq666} to
$\widehat u$, we obtain
\[
 L_r[\widehat u]\leq \widehat{\mathrm{C}}_{1} r^{\frac{2+p}{1+p-\mu}},
 \quad  0<r<1,
\]
where $ \widehat{\mathrm{C}}_{1} $ is a positive constant depending only on $n,\lambda,\Lambda,p,\mu,\mathcal M_1,\mathcal M_2 $ and $ A_0 $. Monotonicity in time gives
\[
 L_r[\widehat u]
 =\sup_{B_r}\widehat u(\cdot,0).
\]
For $0<|x|<1/2$, we apply this estimate with $r>|x|$
and let $r \to |x|$. Together with
$\widehat u(0,0)=0$, this yields
\[
 \widehat u(x,0)\leq  \widehat{\mathrm{C}}_{1}|x|^{\frac{2+p}{1+p-\mu}},
 \quad |x|<1/2.
\]
Returning to the original spatial variable $ u $, we obtain
\begin{equation}\label{Section3:eq-spatial-growth}
 u(y,t_*)\leq \widehat{\mathrm{C}}_{1} \rho^{-\frac{2+p}{1+p-\mu}} |y-x_*|^{\frac{2+p}{1+p-\mu}},
 \quad   |y-x_*|<\rho/2.
\end{equation}

Choose a fixed $0<R<\rho/2 $, sufficiently small that
\[
 Q_R(z_*)\Subset Q_1
 \quad\text{for every }z_*\in Q_{3/4}.
\]
By Proposition~\ref{Section3:prop31}, the estimate
\eqref{Section3:eq-spatial-growth} implies
\begin{equation}
\label{Section3:forward1}
\sup_{B_r(x_*)\times[t_*,t_*+r^\theta]}u
 \leq \mathrm{\widehat{C}} r^{\frac{2+p}{1+p-\mu}},
 \quad  0<r<R.
\end{equation}
For earlier times, time monotonicity and
\eqref{Section3:eq-spatial-growth} give
\begin{equation}
\label{Section3:back1}
 u(y,t)\leq u(y,t_*)\leq \widehat{\mathrm{C}}_{1} \rho^{-\frac{2+p}{1+p-\mu}}r^{\frac{2+p}{1+p-\mu}},
\end{equation}
for $ y\in B_r(x_*)$ and $ t_*-r^\theta<t\leq t_* $. Combining \eqref{Section3:forward1} and \eqref{Section3:back1}, we obtain
\begin{equation}\label{Section3:eq-free-boundary-growth}
\sup_{Q_r(z_*)}u \leq \mathrm C_2r^{\frac{2+p}{1+p-\mu}},
 \quad
 z_*\in\partial\{u>0\}\cap Q_{3/4},  \quad
 0<r<R,
\end{equation}
where
\[
 \mathrm C_2:=\max\left\{\widehat{\mathrm{C}}_{1} \rho^{-\frac{2+p}{1+p-\mu}},\mathrm{\widehat{C}}\right\}.
\]

It remains to relate \eqref{Section3:eq-free-boundary-growth} to $\Gamma$.
Choose $\eta>0$ small enough that $2\eta<R$ and
\[
 Q_{2\eta}(z)\Subset Q_{3/4}
 \quad\text{for every }z\in Q_{1/2}.
\]
Fix $z=(x,t)\in Q_{1/2}$.
If $u(z)=0$, the desired estimate is immediate.
If $\Gamma(z)\geq\eta$, then $u\leq1$ gives
\[
 u(z)\leq1\leq\eta^{-\alpha}\Gamma(z)^\alpha.
\]

We therefore consider the remaining case, in which $ u(z)>0, 0<\gamma:=\Gamma(z)<\eta $.
Here $\gamma>0$ follows from continuity.
Since the admissible cylinders are nested, every
$Q_r(z)$ with $r<\gamma$ is contained in $\{u>0\}$.
Hence $u>0$ in  $ Q_\gamma(z) $.

We claim that $u$ vanishes at some point of
$\partial Q_\gamma(z)$. Indeed,
$\overline{Q_\gamma(z)}$ is compactly contained
in $Q_{3/4}$. If $u$ were positive on this entire
closure, continuity and compactness would make $u$
positive on a neighborhood of it. We could then
increase the radius slightly, contradicting the
definition of $\gamma$. Thus there exists $ z_*=(x_*,t_*)\in\partial Q_\gamma(z) $ such that $ u(z_*)=0 $. Since $u>0$ inside $Q_\gamma(z)$, this point belongs to $\partial\{u>0\}\cap Q_{3/4}$.
Moreover,
\[
 |x-x_*|\leq\gamma,
 \quad |t-t_*|\leq\gamma^\theta,
\]
so $z\in Q_{2\gamma}(z_*)$.
Since $2\gamma<R$, we apply
\eqref{Section3:eq-free-boundary-growth}
at $z_*$ to obtain
\[
 u(z)
 \leq\sup_{Q_{2\gamma}(z_*)}u
 \leq\mathrm C_2(2\gamma)^\alpha
 =2^\alpha\mathrm C_2\Gamma(z)^\alpha.
\]
Finally, combining all cases proves the desired estimate \eqref{Section 3:Thm3.1:main} with $ \mathrm C
 =\max\{\eta^{-\alpha},\,2^\alpha\mathrm C_2\} $. This completes the proof of Theorem~\ref{Sec3:thm1}.    
\end{proof}  
 }

\subsection{Upper growth theorem}
{
We will use Theorem~\ref{Sec3:thm1} to prove Theorem \ref{Thm1}.
\begin{proof}[{\bf Proof of Theorem \ref{Thm1}}]
Fix $ z_{0} = (x_0, t_0) \in\partial\{u>0\}\cap K $, then by continuity, $u(x_0,t_0)=0 $. We divide the proof into three steps:

\smallskip
\noindent
{\bf Step 1. A uniform estimate on the time slice $t=t_0$.} Let
\[
 M:=\max\{1,\|u\|_{L^\infty(\Omega_T)}\},
 \quad
 d_x:=\inf_{(x,t)\in K}\operatorname{dist}(x,\partial \Omega),
 \quad \text{and} \quad 
 t_-:=\min_{(x,t)\in K}t.
\]
Choose
\[
 R_0:=\min\left\{
 1,\frac{d_x}{2},
 \left(\frac{t_-}{2}\right)^{1/(p+2)}
 \right\},
 \quad \text{and} \quad 
 \tau_0:=R_0^{p+2}M^{-p}.
\]
These choices are uniform over $z_0\in K$ and depend only on the spatial distance to the boundary and the distance to the initial time.
Now, we define 
\[
 \omega(x,t):=
 \frac{u(x_0+R_0x,t_0+\tau_0t)}{M},
 \quad (x,t)\in Q_1^-,
\]
then  direct computation shows that $ \omega $ solves
\begin{equation*}
|D\omega|^{p} F_{0}(D^{2}\omega) - \partial_{t} \omega = \lambda_{\omega}(x,t) \omega^{\mu}_{+}     \quad   \text{in}  \quad  Q_{1}^{-},
\end{equation*}
where
\[
 F_0(\mathrm{X})=\frac{R_0^2}{M}
 F\left(\frac{M}{R_0^2}\mathrm{X}\right),    \quad     \text{and}   \quad  \lambda_{\omega}(x,t):=R_{0}^{p+2}M^{-(1+p-\mu)} \lambda_{0}(x_0+R_0x, t_{0}+\tau_0t)
\]
and $ \omega $ still satisfies the condition~\hyperref[AC]{\bf (AC)}. Moreover, $F_0$ satisfies~\hyperref[F1]{\bf (F1)} with the same ellipticity
constants, $ 0\leq \omega \leq1 $ and $ \omega(0,0)=0 $.

By the backward estimate from the proof of
Theorem~\ref{Sec3:thm1}, we obtain
\[
 \omega(x,0)\leq \mathrm{C}|x|^{\frac{2+p}{1+p-\mu}},
 \quad |x|<1/2,
\]
where $\mathrm C >0$ depends only on
$n,\lambda,\Lambda,p,\mu, \mathcal M_1, \mathcal M_2$, $ A_0, M $ and $ R_{0}$. Returning to the original spatial variables $y$, we obtain
\begin{equation}\label{eq:main-slice-growth}
 u(y,t_0)\leq \mathrm{C}MR_0^{-\frac{2+p}{1+p-\mu}}|y-x_0|^{\frac{2+p}{1+p-\mu}},
 \quad  |y-x_0|<R_0/2.
\end{equation}

\smallskip
\noindent
{\bf Step 2. Normalize on a full cylinder.}
We set
\[
 B:=\max\left\{
 1,\,
 \bigl[(\mathrm{C}M)^{1-\mu}R_0^{-\theta}+A_0(1-\mu)\bigr]^{1/(1-\mu)}
 \right\}.
\]
Choose a fixed radius $0<r<R_0/2$, sufficiently small that
\[
 Q_r(z_0)\Subset \Omega_T
 \quad   \text{for every }  z_0 \in K.
\]
We define the rescaling function
\begin{equation*}
v(x,t):=\frac{u(x_0+rx,t_0+\tau t)}{a},
 \quad (x,t)  \in  Q_1,
\end{equation*}
where
\[
 a:=Br^\alpha,\quad
 \tau:=\frac{r^{p+2}}{a^p}
      =B^{-p}r^\theta,  \quad   \text{and} \quad   \alpha:= \frac{2+p}{1+p-\mu}.
\]
Since $B\geq1$, we have $\tau\leq r^\theta$, thus this normalization is defined on the full cylinder.

A direct computation shows that $ v $ satisfies
\[
 |Dv|^p\widetilde F(D^2v)-v_t
 =\widetilde\lambda_0(x,t)v^\mu_{+}
\]
in the viscosity sense, where
\[
 \widetilde F(\mathrm{X})=\frac{r^2}{a}
 F\left(\frac{a}{r^2}\mathrm{X}\right)   \quad  \text{and}  \quad   \widetilde\lambda_0(x,t)
 :=B^{-(1+p-\mu)}\lambda_0(x_0+rx,t_0+\tau t).
\]
Likewise, \hyperref[AC]{\bf (AC)} becomes $ 0\leq v_t\leq B^{-(1+p-\mu)}A_0v^\mu $ in the viscosity sense.
The operator $\widetilde F$ satisfies~\hyperref[F1]{\bf (F1)} with the same
ellipticity constants. Furthermore,
\[
 0<B^{-(1+p-\mu)}\mathcal M_1
 \leq\widetilde\lambda_0\leq\mathcal M_2.
\]

We verify that $0\leq v\leq1$.
For $-1<t\leq0$, time monotonicity, the definition of $ a $ and \eqref{eq:main-slice-growth} yield
\begin{equation*}
u(x_0+rx,t_0+\tau t)
 \leq u(x_0+rx,t_0)
 \leq \mathrm{C}MR_0^{-\alpha}r^\alpha
 \leq a.
\end{equation*}
For $0\leq t<1$, the temporal estimate following from Lemma~\ref{Section 2:lemma2}
gives
\[
 \begin{aligned}
 u(x_0+rx,t_0+\tau t)^{1-\mu}
 &\leq u(x_0+rx,t_0)^{1-\mu}
       +A_0(1-\mu)\tau t\\
 &\leq
 \bigl[(\mathrm{C}MR_0^{-\alpha})^{1-\mu}+A_0(1-\mu)\bigr]r^\theta\\
 &\leq B^{1-\mu}r^\theta
 =a^{1-\mu}.
 \end{aligned}
\]
Consequently, $0\leq v\leq1$ in $Q_1$ and $ v(0,0)=0 $.
Thus $v$ satisfies the hypotheses of
Theorem~\ref{Sec3:thm1} with uniform upper bounds.

\smallskip
\noindent
{\bf Step 3. Conclusion.} Applying Theorem~\ref{Sec3:thm1} to $ v $, we have
\begin{equation}\label{Section3.3:eq11}
 v(x,t)\leq \mathrm{C}_N\Gamma(x,t)^\alpha,
 \quad   (x,t)\in Q_{1/2},
\end{equation}
where
$\mathrm{C}_N=\mathrm{C}_N(n,\lambda,\Lambda,p,\mu,\mathcal M_1,\mathcal M_2,A_0, \mathrm{C}, M, R_{0}) $.

Since $v(0,0)=0$, any cylinder centered at $(x,t)$
with radius larger than
$\max\{|x|,|t|^{1/\theta}\}$ contains a zero of $ v $.
Therefore,
\begin{equation}\label{Section3.3:eq12}
 \Gamma(x,t)
 \leq\max\{|x|,|t|^{1/\theta}\}
 \leq |x|+|t|^{1/\theta}.
\end{equation}
We combine \eqref{Section3.3:eq11} and \eqref{Section3.3:eq12} to obtain
\[
 v(x,t)\leq
 \mathrm{C}_N\bigl(|x|+|t|^{1/\theta}\bigr)^\alpha,
 \quad    (x,t)\in Q_{1/2}.
\]
For a point in the original variables $(y,\sigma)$, write $ d:=\operatorname{dist}_p((y,\sigma),z_0)$. If
\[
 d<\frac{r}{2B^{p/\theta}},
\]
then
\[
 \left(\frac{y-x_0}{r},
       \frac{\sigma-t_0}{\tau}\right)\in Q_{1/2}.
\]
Returning to the original variables gives
\[
 \begin{aligned}
 u(y,\sigma)
 &\leq \mathrm{C}_Na
 \left(
 \frac{|y-x_0|}{r}
 +\frac{|\sigma-t_0|^{1/\theta}}{\tau^{1/\theta}}
 \right)^\alpha\\
 &=\mathrm{C}_NB
 \left(
 |y-x_0|+
 B^{p/\theta}|\sigma-t_0|^{1/\theta}
 \right)^\alpha\\
 &\leq \mathrm{C}_NB^{1+\frac{p\alpha}{\theta}}
 \operatorname{dist}_p((y,\sigma),z_0)^\alpha   \\
 & = \mathrm{C}_NB^{\frac{1+p-\mu}{1-\mu}}\operatorname{dist}_p((y,\sigma),z_0)^\alpha,
 \end{aligned}
\]
where we have used
\[
 1+\frac{p\alpha}{\theta}
 =\frac{1+p-\mu}{1-\mu}
\]
in the last equality. Consequently, we obtain the desired estimate
\begin{equation*}
 u(y,\sigma) \leq  \mathrm C_0 \operatorname{dist}_p((y,\sigma),(x_0, t_0))^{\frac{2+p}{1+p-\mu}},
\end{equation*}
where $ \mathrm C_0:=\mathrm{C}_NB^{\frac{1+p-\mu}{1-\mu}} $, depends only on $ n, \lambda, \Lambda, \mathcal{M}_{1}, \mathcal{M}_{2}, p, \mu, \|u\|_{L^{\infty}(\Omega_{T})}, A_{0}$ and $ \mathrm{dist}(K, \partial_{\mathrm{par}} \Omega_{T}) $. This completes the proof of~Theorem~\ref{Thm1}.   
\end{proof}
}

\vspace{3mm}

\section{Geometric measure properties and other consequences}\label{Section 4}

In this subsection, we derive, from Theorems~\ref{Thm1} and~\ref{Thm2}, consequences for the measure-theoretic properties of the free boundary and {\it finite speed propagation}. By a standard argument (cf.~\cite[Corollary 4.8]{SO19}), Theorems~\ref{Thm1} and~\ref{Thm2} yield the following positive-density property:
\begin{corollary}[{\it Positive Lebesgue density of $ \{u>0\}$}]\label{Sec1:coro2}
Suppose that $ F $ satisfies \hyperref[F1]{\bf (F1)}. Let $ u $ be the viscosity solution to \eqref{DCP}, {so that \hyperref[AC]{\bf (AC)} holds in $ Q_{1}$}. Then, there exists a positive constant $ \zeta
=\zeta\bigl(
n,p,\mu,\lambda,\Lambda,\mathcal M_1,\mathcal M_2,
A_0,\|u\|_{L^\infty(Q_1)}
\bigr)\in(0,1). $ such that, for every $ (x_{0}, t_{0}) \in  \overline{\{u>0\}} $ and $ 0 < r <1 $ such that $ Q_{r}(x_{0}, t_{0}) \subset Q_{1/2} $, the inclusion
 \begin{equation*}
   Q_{\zeta r}(\widetilde{x}, \widetilde{t}) \subset Q_{r}(x_{0}, t_{0}) \cap \{u>0\}
 \end{equation*}
holds for some $ (\widetilde{x}, \widetilde{t}) \in Q_{r}^{-}(x_{0}, t_{0}) $.
\end{corollary}

We observe that Corollary \ref{Sec1:coro2} guarantees that the free boundary cannot have Lebesgue points. In other words, $  \mathcal{L}^{n+1}(\partial \{u>0\}\cap K) = 0    $ for any compact set $ K \subset Q_{1} $.

Next, we shall provide, as a straightforward consequence of the previous results,
that the free boundary is a porous set. We recall the definition of this notion.

\begin{definition}[{\it Porous set}]\label{def:porous}
A set $E \subset \mathbb{R}^n$ is said to be \emph{porous} with porosity constant
$0 < \delta \leq 1$ if there exists $R > 0$ such that, for each $x_0 \in E$ and
$0 < r < R$, there exists a point $y_0$ satisfying
\[
B_{\delta r}(y_0) \subset B_r(x_0) \setminus E.
\]
\end{definition}

We observe that a porous set has Hausdorff dimension at most $n - c_{0}\,\delta^{n}$,
where $c_{0} = c_{0}(n) > 0 $. This proof is standard and relies chiefly on Theorems~\ref{Thm2} and \ref{Sec3:thm1} (cf. \cite[Corollary 4.11]{SO19}). We omit the details for brevity.

\begin{corollary}[{\it Porosity of $t-$level free boundaries}]\label{cor:porosity}
Suppose that $ F $ satisfies \hyperref[F1]{\bf (F1)} and {\hyperref[AC]{\bf (AC)} holds in $ Q_{1}$}. Let $u$ be a viscosity solution to \eqref{DCP}. For every compact set
$K \Subset Q_{1}$, one has
\[
\mathscr{H}^{n-\varepsilon}
\bigl(\partial\{u > 0\} \cap K \cap \{t = t_{0}\}\bigr) < \infty,
\]
for some constant $0<\varepsilon
=\varepsilon\bigl(
n,p,\mu,\lambda,\Lambda,\mathcal M_1,\mathcal M_2,
A_0,\|u\|_{L^\infty(Q_1)},
\operatorname{dist}(K,\partial Q_1)
\bigr)\le1$.
\end{corollary}

In contrast to the classical heat equation---which is characterized by its
infinite speed of propagation---fully nonlinear degenerate parabolic dead-core problems
exhibit a {\it finite speed of propagation}. This feature enhances the physical
relevance of such equations in modeling diffusive processes. Moreover, the
emergence of this phenomenon stems from the loss of diffusivity of the equation
along the level set $\{u = 0\}$. {The proof follows from the barrier argument used in the proof of Theorem~\ref{Thm2}  (cf. \cite[Corollary 4.4]{CW03}), together with the time control provided by condition~\hyperref[AC]{\textbf{(AC)}}.} In contrast to the standard arguments, it suffices to modify the growth behavior of $ u $ to $ r^{\frac{2+p}{1+p-\mu}} $.
\begin{corollary}[{\it Finite speed propagation of $\{u>0\}$}]
Suppose that $F$ satisfies \hyperref[F1]{\textbf{(F1)}}.
There exists a constant {$\mathrm{c}
=\mathrm{c}(n,p,\mu,\lambda,\Lambda,\mathcal M_1,A_0)>1 $} such that, for any nonnegative viscosity solution $u$ to
\eqref{DCP} {satisfying \hyperref[AC]{\textbf{(AC)}}}
in $Q\times(0,\infty)$, and any
$Q_r^+(x_0,t_0)\subset Q\times(0,\infty)$,
the following implication holds:
\[
u(\cdot,t_0)=0 \quad\text{in }B_r(x_0)
\quad\Longrightarrow\quad
u(\cdot,t_0+s^\theta)=0
\quad\text{in }B_{r-\mathrm{c}s}(x_0)
\]
for every $0<s<r/\mathrm{c}$.
\end{corollary}

\vspace{3mm}

\section{Gradient decay and global analysis results}\label{Section 5}

In this section, we first establish a gradient decay estimate for the class of equations \eqref{DCP}. We then prove a Liouville-type result for global dead-core solutions and carry out a blow-up analysis at free boundary points.

\subsection{Gradient decay}
In what follows, we establish the sharp gradient decay estimate for viscosity solutions of \eqref{DCP}. Combining the interior $C^{1,\alpha}$ regularity with the upper growth estimate established in Theorem~\ref{Thm1}, we show that $|Du|$ vanishes at the optimal rate $\operatorname{dist}_{p}(\cdot,\partial\{u>0\})^{\frac{1+\mu}{1+p-\mu}}$.  The proof is based on a rescaling argument and is organized as follows:    

{
\begin{proof}[{\bf Proof of Theorem~\ref{Thm3}}]
We first choose a uniform radius for applying the upper
growth estimate. Set $K:=\overline{Q_{3/4}}$.
The proof of Theorem~\ref{Thm1}, with the localization
parameters chosen uniformly on $K$, provides constants
$ \mathrm{C}_*>0$ and $\rho_*>0$ such that
\[
u(x,t)
\le \mathrm{C}_*
\operatorname{dist}_p\bigl((x,t),(x_0,t_0)\bigr)^{
\frac{2+p}{1+p-\mu}}
\]
whenever
\[
(x_0,t_0)\in\partial\{u>0\}\cap K,
\quad
(x,t)\in Q_1,
\quad
\operatorname{dist}_p\bigl((x,t),(x_0,t_0)\bigr)<\rho_*.
\]
We choose
\[
r_0:=\min\left\{\frac1{16},\frac{\rho_*}{4}\right\}.
\]

Fix $(z,s)\in\overline{\{u>0\}}\cap Q_{1/2}$ and set $ r:=\operatorname{dist}_p
\bigl((z,s),\partial\{u>0\}\bigr) $. If $r=0$, then $ (z,s) $ is a free boundary point and $u(z,s)=0$. By the upper growth estimate (Theorem~\ref{Thm1}), we have
\[
0\le u(z+h,s) \leq \mathrm{C}_*|h|^{\frac{2+p}{1+p-\mu}}
\]
for all sufficiently small $h$.
Since
\[
\frac{2+p}{1+p-\mu}>1,
\]
we obtain $Du(z,s)=0$.
We may therefore assume that $r>0$ and distinguish
two cases.

\smallskip
\noindent
\textbf{Case 1. $r\ge r_0$.}
Since
\[
\|\lambda_0u^\mu\|_{L^\infty(Q_1)}
\le
\mathcal M_2\|u\|_{L^\infty(Q_1)}^\mu,
\]
we apply the nonhomogeneous interior gradient estimate discussed in Remark~\ref{Sec2:rmk2}, based on
Proposition~\ref{Se2:Pro1}, to obtain
\[
|Du(z,s)|\leq M=M\bigl(
n,\lambda,\Lambda,p,\mu,\mathcal M_2,
\|u\|_{L^\infty(Q_1)}
\bigr).
\]
Combining this bound with $ r\ge r_0$, yields
\[
|Du(z,s)|
\le M
\le M r_0^{-\frac{1+\mu}{1+p-\mu}}
       r^{\frac{1+\mu}{1+p-\mu}}.
\]

\smallskip
\noindent
\textbf{Case 2. $0<r<r_0$.}
Since $(z,s)\in Q_{1/2}$ and $r<\frac{1}{16}$, a sequence approaching the infimum defining the distance to the free boundary can be
chosen in a compact subset of $Q_1$.
Hence there exists
$(x_0,t_0)\in\partial\{u>0\}\cap Q_{3/4}$ such that
\[
r=|z-x_0|+|s-t_0|^{\frac{1}{\theta}}.
\]
For every $(x,t)\in Q_r(z,s)$, we have
\[
|x-x_0|
\le |x-z|+|z-x_0|<2r
\quad 
\text{and}
\quad 
|t-t_0|
\le |t-s|+|s-t_0|
<2r^\theta\le(2r)^\theta,
\]
provided $ \theta \geq 1 $, which is equivalent to $ \mu \leq \frac{1}{1+p} $. Therefore,
\begin{equation}\label{Section5.1:eq890}
  Q_r(z,s)\subset Q_{2r}(x_0,t_0).
\end{equation}
The choice $r_0\leq \frac{1}{16} $ also ensures that $ Q_{2r}(x_0,t_0)\Subset Q_{3/4}$. Moreover, every point $(x,t)\in Q_{2r}(x_0,t_0)$
satisfies
\[
\operatorname{dist}_p\bigl((x,t),(x_0,t_0)\bigr)
<4r<\rho_*.
\]
Combining \eqref{Section5.1:eq890} and the upper growth estimate (Theorem~\ref{Thm1}),  yields
\begin{equation}\label{eq:gradient-local-growth}
\sup_{Q_r(z,s)}u
\le \sup_{Q_{2r}(x_0,t_0)}u
\le \mathrm{C}_1r^{\frac{2+p}{1+p-\mu}},
\end{equation}
where $ \mathrm{C}_1:=4^{\frac{2+p}{1+p-\mu}} \mathrm{C}_{*} $ is independent of $ r $.    

We define the rescaled function
\[
\omega(x,t):=
\frac{u(z+rx,s+r^\theta t)}
     {r^{\frac{2+p}{1+p-\mu}}},
\quad (x,t)\in Q_1.
\]
A direct calculation shows that $\omega$ satisfies
\[
|D\omega|^p\widetilde F(D^2\omega)-\omega_t
=
\widetilde\lambda_0(x,t)\omega^\mu
\quad\text{in }Q_1
\]
in the viscosity sense, where
\[
\widetilde F(\mathrm{X}):=
r^{\frac{p-2\mu}{1+p-\mu}}
F\left(r^{\frac{2\mu-p}{1+p-\mu}}\mathrm{X}\right)
\quad 
\text{and}
\quad 
\widetilde\lambda_0(x,t):=
\lambda_0(z+rx,s+r^\theta t).
\]
It is straightforward to verify that uniform ellipticity (\hyperref[F1]{\textbf{(F1)}}), and convexity/concavity (\hyperref[F2]{\textbf{(F2)}}) are preserved by this scaling. Furthermore, the temporal condition~\hyperref[AC]{\bf (AC)} is preserved as well. Indeed,
\[
0\le\omega_t
\le
A_0r^{\theta-\frac{(2+p)(1-\mu)}{1+p-\mu}}
\omega^\mu
=A_0\omega^\mu
\]
in the viscosity sense, by the definition of $\theta$.

By \eqref{eq:gradient-local-growth}, we have
\[
0\le\omega\le \mathrm{C}_1
\quad\text{in }Q_1.
\]
Consequently,
\[
\|\widetilde\lambda_0\omega^\mu\|_{L^\infty(Q_1)}
\le\mathcal M_2 \mathrm{C}_1^\mu.
\]
Thus both $\omega$ and its source term are bounded
independently of $r$. Applying the gradient estimate for bounded solutions from Proposition~\ref{Se2:Pro1} and Remark~\ref{Sec2:rmk2}(i), we obtain
\[
|D\omega(0,0)|\leq \mathrm{C}_2,
\]
where $\mathrm{C}_2 >0$ is independent of $ r $. Returning to the original variables, we conclude that
\[
|Du(z,s)|
=
r^{\frac{1+\mu}{1+p-\mu}}|D\omega(0,0)|
\le
\mathrm{C}_2r^{\frac{1+\mu}{1+p-\mu}}.
\]
Finally, combining the two cases gives the desired estimate
\begin{equation*}
  |Du(z,s)| \leq \mathrm{C} r^{\frac{1+\mu}{1+p-\mu}},
\end{equation*}
where $ \mathrm{C}= \mathrm{C}\bigl(
n,\lambda,\Lambda,\mathcal M_1,\mathcal M_2,
p,\mu,A_0,\|u\|_{L^\infty(Q_1)}
\bigr) $ is a positive constant.
\end{proof}
}

\subsection{Liouville theorem}

In this subsection, with Theorem~\ref{Thm2} in hand, we proceed to prove Theorem~\ref{Thm4}.   

{\begin{proof}[{\bf Proof of Theorem~\ref{Thm4}}]
For simplicity, we denote $\varrho(x,t):=\max\{|x|,|t|^{1/\theta}\}$.
We first show that:
\begin{equation}\label{Sec5:Liouville-upper}
 \sup_{Q_R}u=o(R^\frac{2+p}{1+p-\mu}) \quad \text{as}\quad R\to\infty.
\end{equation}
Indeed, given $\eta>0$, choose $L>0$ such that
$u(x,t)\le\eta\varrho(x,t)^\frac{2+p}{1+p-\mu}$ whenever $\varrho(x,t)\ge L$.
By continuity of $ u$, we find $\displaystyle M_L:=\max_{\overline{Q_L}}u$ is finite. Hence,
\[
 \sup_{Q_R}u\le\max\{M_L,\eta R^\frac{2+p}{1+p-\mu}\}
 \quad \text{for }R\ge L.
\]
Dividing by $R^\frac{2+p}{1+p-\mu} $, first letting $R\to\infty$, and then letting $\eta \to 0$, we obtain \eqref{Sec5:Liouville-upper}.

Suppose, by contradiction, that $u(x_*,t_*)>0$ for some point $ (x_*,t_*) \in \mathbb{R}^{n} \times \mathbb{R} $.
Since $u$ is an entire solution, every finite cylinder centered at
$(x_*,t_*)$ is admissible in Theorem~\ref{Thm2}. By using Theorem~\ref{Thm2}, we have
\begin{equation}\label{Sec5:Liouville-lower}
 \mathrm{C}_{0}^{*}R^\frac{2+p}{1+p-\mu}
 \le\sup_{\partial_{\mathrm{par}}Q_R^-(x_*,t_*)}u
 \quad\text{for every }R>0,
\end{equation}
where $\mathrm{C}_{0}^{*} >0$ is independent of $R$.

It remains to compare this lower bound \eqref{Sec5:Liouville-lower} with \eqref{Sec5:Liouville-upper}.
Set $a_\theta:=2+2^{1/\theta}$.
For $R\ge\max\{1,|x_*|,|t_*|^{1/\theta}\}$, every point
$(x,t)\in\overline{Q_R^-(x_*,t_*)}$ satisfies
\[
 |x|\le |x_*|+R\le2R,
 \quad    \text{and}  \quad
 |t|\le |t_*|+R^\theta\le2R^\theta.
\]
Since $a_\theta>2$ and $a_\theta^\theta>2$, then one readily checks that
$\overline{Q_R^-(x_*,t_*)}\subset Q_{a_\theta R}$.
We combine this inclusion with
\eqref{Sec5:Liouville-upper} and \eqref{Sec5:Liouville-lower} to obtain
\begin{equation}\label{Sec5:Liouville-lower00}
\mathrm{C}_{0}^{*}R^\frac{2+p}{1+p-\mu}
 \le\sup_{Q_{a_\theta R}}u
 =o(R^\frac{2+p}{1+p-\mu}).
\end{equation}
Dividing by $R^\frac{2+p}{1+p-\mu}$, the right-hand side of \eqref{Sec5:Liouville-lower00} tends to zero, contradicting $ \mathrm{C}_{0}^{*} >0$.
Thus \(u\le0\) everywhere. Since $u\ge0$, we conclude that $u\equiv0$.
\end{proof}
}

\subsection{Blow-up analysis}

In this subsection, we investigate the blow-up behavior near free boundary points.
Let $u$ be a solution to \eqref{DCP}, and let $(x_{0}, t_{0}) \in \partial \{u>0\} \cap Q_{1/2}$. Choose $ \rho > 0 $ such that $ \overline{Q_{\rho}(x_0, t_0)}  \subset  Q_1 $. For each $\epsilon \searrow 0$, we define the blow-up sequence $u_{\epsilon}: Q_{\frac{\rho}{\epsilon}} \rightarrow \mathbb{R}$ by
\begin{equation*}
  u_{\epsilon}(x,t):= \frac{u(x_{0}+\epsilon x, t_{0}+\epsilon^{\theta}t)}{\epsilon^{\frac{2+p}{1+p-\mu}}}.
\end{equation*}

We now state the main result of this subsection.

{\begin{theorem}[{\it Blow-up limit}]\label{Thm5.1-B-L}
Suppose that $ F $ satisfies~\hyperref[F1]{\bf (F1)} and \hyperref[F2]{\bf (F2)}. Let $u$ be a nonnegative bounded viscosity solution to \eqref{DCP}, so that \hyperref[AC]{\bf (AC)} holds in $ Q_{1} $. Fix a free boundary point $(x_{0}, t_{0}) \in \partial \{u>0\} \cap Q_{1/2}$. Then, up to a subsequence,
\begin{equation*}
  u_{\epsilon} \rightarrow u_{0} \quad \text{and} \quad Du_{\epsilon} \rightarrow Du_{0}
  \quad \text{uniformly on every compact set} \ K \Subset \mathbb{R}^{n} \times \mathbb{R}.
\end{equation*}
Moreover, the blow-up limit $u_{0}$ satisfies $u_{0}(0,0)=0$ and
\begin{equation*}
  |Du_{0}|^{p}  \widetilde{F}_0(D^{2}u_{0}) - \partial_{t} u_{0} = \lambda_{0}(x_{0}, t_{0}) (u_{0})^{\mu}_{+}(x,t)
  \quad \text{in} \quad \mathbb{R}^{n} \times \mathbb{R}
\end{equation*}
in the viscosity sense. It also satisfies $ 0 \leq (u_0)_{t} \leq A_0u_{0}^{\mu} $ in the viscosity sense. In addition, there exist constants $\mathrm{c}_{1}, \mathrm{c}_{2} > 0$, independent of $ R $, such that
\begin{equation}\label{Section5:thm5.1-eq10}
  \mathrm{c}_{1} R^{\frac{2+p}{1+p-\mu}} \leq \sup_{Q_{R}^{-}} u_{0} \leq \sup_{Q_{R}} u_{0}  \leq \mathrm{c}_{2} R^{\frac{2+p}{1+p-\mu}}
\end{equation}
for every $ R > 0 $.
\end{theorem}

\begin{proof}
We divide the proof into two steps:

\smallskip
\noindent
{\bf Step 1.} By the choice of $ \rho $, the function $ u_{\epsilon} $ is well defined on $ Q_{\frac{\rho}{\epsilon}} $. We observe that $u_{\epsilon}$ satisfies, in the viscosity sense,
\begin{equation}\label{Section5:thm5.1-eq1}
  |Du_{\epsilon}|^{p} \widetilde{F}_{\epsilon}(D^{2}u_{\epsilon}) - (u_{\epsilon})_{t}
  = \widetilde{\lambda}_{0,\epsilon}(x,t) (u_{\epsilon}^{\mu})_{+}
  \quad \text{in} \quad Q_{\frac{\rho}{\epsilon}},
\end{equation}
where
\begin{equation*}
  \widetilde{F}_{\epsilon}(\mathrm{X})
  := \epsilon^{\frac{p-2\mu}{1+p-\mu}} F(\epsilon^{\frac{2\mu-p}{1+p-\mu}}\mathrm{X}),
  \qquad \text{and} \qquad
  \widetilde{\lambda}_{0,\epsilon}(x,t) := \lambda_{0}(x_{0}+\epsilon x, t_{0}+\epsilon^{\theta}t).
\end{equation*}
The rescaled operator $\widetilde{F}_{\epsilon}$ also satisfies  \hyperref[F1]{\bf (F1)} and \hyperref[F2]{\bf (F2)} with the same structural constants. Moreover, we have \begin{equation*}
 0 \leq (u_{\epsilon})_{t} \leq A_0 u_{\epsilon}^{\mu}  \quad   \text{in the viscosity sense}.
\end{equation*}

By Theorem~\ref{Thm1}, applied in the fixed interior cylinder
$Q_{3/4}\Subset Q_1$, there exist constants $\mathrm{C}_1>0$
and $r>0$, independent of the choice of
$z_0:=(x_0,t_0)\in\partial\{u>0\}\cap Q_{1/2}$, such that
\begin{equation*}
 u(y,s) \leq \mathrm{C}_{1} \mathrm{dist}_{p}((y,s), z_{0})^{\frac{2+p}{1+p-\mu}}, \quad   \text{whenever}   \quad  \mathrm{dist}_{p}((y,s), z_{0}) < r.
\end{equation*}
Fix $R>0$ and choose $\epsilon>0$ sufficiently small that
$4\epsilon R<\rho$ and $8\epsilon R<r$.
Then $\overline{Q_{4R}}\subset Q_{\rho/\epsilon}$, and
\begin{equation}\label{Section5:thm5.1-eq2}
0 \leq u_{\epsilon}(x,t) \leq \mathrm{C}_{1}(|x|+|t|^{\frac{1}{\theta}})^{\frac{2+p}{1+p-\mu}} \leq \mathrm{C}_{1} (8R)^{\frac{2+p}{1+p-\mu}} \quad   \text{in} \quad \overline{Q_{4R}}.
\end{equation}

In view of \eqref{Section5:thm5.1-eq2},
Proposition~\ref{Se2:Pro1} and Remark~\ref{Sec2:rmk2}(i)
yield uniform local spatial $C^{1,\beta}$ bounds
for $u_\epsilon$, for some $\beta>0$.
Combined with the uniform Lipschitz bounds in time obtained
from \hyperref[AC]{\textbf{(AC)}} and
Lemma~\ref{Section 2:lemma2}, these estimates yield,
by interpolation, local equicontinuity in time
of the family $\{Du_\epsilon\}$.
Thus, by the Arzel\`a--Ascoli theorem and a diagonal argument,
up to a subsequence,
\[
u_{\epsilon}\to u_0
\quad\text{and}\quad
Du_{\epsilon}\to Du_0
\]
locally uniformly in $\mathbb R^n\times\mathbb R$. Moreover, by uniform ellipticity (\hyperref[F1]{\bf (F1)}), after passing to a further subsequence, we have $\widetilde{F}_{\epsilon} \to \widetilde{F}_{0} $ locally uniformly in $ \mathrm{Sym}(n) $. By the stability of viscosity solutions
(Theorem~\ref{Section 2: Thm2}), passing to the limit
in \eqref{Section5:thm5.1-eq1} shows that $u_0$ satisfies
\begin{equation*}
  |Du_{0}|^{p} \widetilde{F}_0(D^{2}u_{0}) - (u_{0})_{t}
  = \lambda_{0}(x_{0}, t_{0})(u_{0})^{\mu}_{+}(x,t)
  \quad \text{in} \quad \mathbb{R}^{n} \times \mathbb{R}.
\end{equation*}
Furthermore, we obtain $ 0 \leq (u_{0})_{t} \leq A_0 u_{0}^{\mu} $ (in the viscosity sense) and $ u_{0}(0,0)=0 $.

\smallskip
\noindent
{\bf Step 2.} Fix $R>0$. The growth bound in
\eqref{Section5:thm5.1-eq2} and local uniform convergence yield
\begin{equation*}
\sup_{Q_{R}} u_{0} \leq \mathrm{C}_{1}(2R)^{\frac{2+p}{1+p-\mu}}.
\end{equation*}
This proves the upper bound in
\eqref{Section5:thm5.1-eq10}. For the lower bound, applying Theorem~\ref{Thm2}
to $u$ in $Q_{\epsilon_jR}^{-}(x_0,t_0)$ yields
\begin{equation*}
\sup_{\partial_{\mathrm{par}}Q_{\epsilon_{j}R}^{-}(x_0, t_0)}u \geq  \mathrm{C}_{0}^{*}(\epsilon_{j}R)^{\frac{2+p}{1+p-\mu}}
\end{equation*}
for all sufficiently large $j$ (depending on $R$),
where $\mathrm{C}_0^*>0$ is independent of $j$ and $R$. After scaling, this becomes
\begin{equation*}
\sup_{\partial_{\mathrm{par}}Q_{R}^{-}} u_{\epsilon_{j}} \geq \mathrm{C}_{0}^{*} R^{\frac{2+p}{1+p-\mu}}.
\end{equation*}
Passing to the limit, using uniform convergence
on $\overline{Q_R^-}$, we obtain
\begin{equation*}
\sup_{\overline{Q_{R}^{-}}} u_{0}  \geq \mathrm{C}_{0}^{*} R^{\frac{2+p}{1+p-\mu}}.
\end{equation*}
Consequently, we obtain the desired estimate \eqref{Section5:thm5.1-eq10} with $ \mathrm{c}_{1}:= \mathrm{C}_{0}^{*} $ and $ \mathrm{c}_{2}:= 2^{\frac{2+p}{1+p-\mu}} \mathrm{C}_{1} $. This completes the proof of Theorem \ref{Thm5.1-B-L}.
\end{proof}
}

\begin{remark}

Motivated by the analysis of the fully nonlinear Alt-Phillips equation, Wu and Yu \cite{WH2022} assumed, in addition to the ellipticity condition \hyperref[F1]{\bf (F1)} that the operator \( F: \mathrm{Sym}(n) \to \mathbb{R} \) satisfies the following structural conditions:   
\begin{equation}\label{Condi-WuYu}\tag{{\bf SC}}
\left\{
\begin{array}{l}
F \text{ is convex,} \\
F(\mathbf{O}_n) = 0, \\
\text{the trace operator belongs to the subdifferential of } F \text{ at } \mathbf{O}_n.
\end{array}
\right.
\end{equation}
Here, $\partial F(\mathfrak{A})$ denotes the {\bf subdifferential}
of $F$ at $\mathfrak{A}$, namely the set of linear functionals
$\mathcal{S}_{\mathfrak{A}}:\mathrm{Sym}(n)\to\mathbb{R}$
satisfying
\[
\mathcal{S}_{\mathfrak{A}}(\mathrm{X})
\le F(\mathfrak{A}+\mathrm{X})-F(\mathfrak{A})
\quad \text{for all }\mathrm{X}\in\mathrm{Sym}(n).
\]

Assume, in addition to \eqref{Condi-WuYu}, that $F$ is differentiable
at $\mathbf{O}_n$. Since the trace operator belongs to
$\partial F(\mathbf{O}_n)$, we have
\[
DF(\mathbf{O}_n)[\mathrm{X}]=\operatorname{tr}(\mathrm{X})
\quad\text{for every } \mathrm{X} \in\operatorname{Sym}(n).
\]
If $\mu>\frac{p}{2}$, then
\[
\widetilde F_\epsilon(\mathrm{X})
=
\epsilon^{\frac{p-2\mu}{1+p-\mu}}
F\left(
\epsilon^{\frac{2\mu-p}{1+p-\mu}}\mathrm{X}
\right)
\longrightarrow \operatorname{tr}(\mathrm{X})
\]
locally uniformly in $\operatorname{Sym}(n)$ as
$\epsilon\to0$.
Therefore, the blow-up limit in Theorem~\ref{Thm5.1-B-L} satisfies
\[
|Du_0|^p\Delta u_0-\partial_tu_0
=
\lambda_0(x_0,t_0)(u_0)_+^\mu
\quad\text{in }\mathbb R^n\times\mathbb R
\]
in the viscosity sense.
\end{remark}

\begin{remark}
When $F(\mathrm{X})=\operatorname{tr}(\mathrm{X})$, nontrivial time-independent
blow-up solutions are given by
\[
u_0(x,t)
=
\mathrm{C}^{1}_{p,\mu,\lambda_0}
(x_i)_\pm^{\frac{2+p}{1+p-\mu}},
\qquad i=1,\dots,n,
\]
where $(x_i)_\pm:=\max\{\pm x_i,0\}$ and
\[
\mathrm{C}^{1}_{p,\mu,\lambda_0}
:=
\left[
\frac{\lambda_0(x_0,t_0)(1+p-\mu)^{2+p}}
{(2+p)^{1+p}(1+\mu)}
\right]^{\frac{1}{1+p-\mu}}.
\]
Another example is the radial profile
\[
u_0(x,t)
=
\mathrm{C}^{\mathrm{rad}}_{n,p,\mu,\lambda_0}
|x|^{\frac{2+p}{1+p-\mu}},
\]
where
\[
\mathrm{C}^{\mathrm{rad}}_{n,p,\mu,\lambda_0}
:=
\left[
\frac{\lambda_0(x_0,t_0)(1+p-\mu)^{2+p}}
{(2+p)^{1+p}
\bigl[2+p+(n-2)(1+p-\mu)\bigr]}
\right]^{\frac{1}{1+p-\mu}}.
\]
\end{remark}

\vspace{3mm}

\section{Discussions and future directions}\label{Section 6}

In this section, we discuss the methodology underlying
the proofs of our main results.
The analytical strategy adopted here is sufficiently flexible
to yield an application related to regularity
at free boundary points. 
In particular, the proof of Corollary~\ref{Sec6:Coro2}
adapts the dyadic iteration scheme used in the proof of
Theorem~\ref{Thm1} and combines it with suitable auxiliary
estimates.   

\subsection{Discussions}
To this end, we shall develop a version of the gradient estimate that is different from Theorem~\ref{Thm3}.    
To establish the gradient estimate, we need to restrict our attention to the following class of functions:    
\begin{definition}
We call $ u \in \mathcal{T}_{p}(Q_{1}) $ if
\begin{enumerate}[i)]

\item   \label{Section6:de1a}  $ u $ is a viscosity solution to
\begin{equation*}
|Du|^{p} F(D^{2}u) - u_{t} = f \in C^{0}(\overline{Q_{1}}) \cap L^{\infty}(Q_{1})
\end{equation*}
and satisfies the following a priori estimate:
$$
\|Du\|_{L^{\infty}(Q_{1/2})} \leq \mathrm{C}_{T}(\|u\|_{L^{\infty}(Q_1)}+\|f\|^{\frac{1}{1+p}}_{L^{\infty}(Q_1)}),
$$
where $ \mathrm{C}_{T} > 0 $ is a universal constant;

\item   \label{Section6:de1b} The estimate in \ref{Section6:de1a}) holds with the same constant $\mathrm C_T$
after any translation and intrinsic rescaling of $u$,
provided that the corresponding cylinder lies in $Q_1$;

\item   \label{Section6:de1c}
There exists a sequence of positive constants
$\{\gamma_j\}_{j\ge1}$, independent of $u$, with
$\gamma_*:=\liminf_{j\to\infty}\gamma_j>0$,
such that, for every $j\ge1$ satisfying
$L_{1/2^{j+1}}[|Du|]>0$ and
$\alpha_j2^{-\theta}\le1$, we have
\[
\sup_{B_{\frac{1}{2^{j+1}}}\times \big(-\alpha_{j}(\frac{1}{2})^{\theta},0\big]} |D u(x,t)| \geq \gamma_{j}L_{\frac{1}{2^{j+1}}}[|D u(x,t)|]. 
\]
where $\alpha_j:=
\frac{1}{2^{2j}}
\left(\frac{1}{L_{1/2^{j+1}}[|Du|]}\right)^p$. 
\end{enumerate}
\end{definition}

We now present a new formulation of the gradient decay property.
    
\begin{corollary}[{\it Gradient decay}]              
\label{Sec6:Coro2}
Suppose that $ F $ satisfies \hyperref[F1]{\bf (F1)} and $ \mathcal{T}_{p}(Q_{1}) \neq \emptyset$. Then there exists a positive constant $ \widetilde{\mathrm{C}}_{1}^{*} = \widetilde{\mathrm{C}}_{1}^{*}(n, \lambda, \Lambda, p, \mu, \mathcal{M}_{1}, \mathcal{M}_{2}, A_0, \mathrm{C}_{T}, \gamma_{*}) $ such that, for every $ u \in \mathbb{J}(F, \lambda_{0}, \mu, A_0)(Q_{1}) \cap \mathcal{T}_{p}(Q_{1})$, there holds
\begin{equation*}
  |D u(x,t)| \leq \widetilde{\mathrm{C}}_{1}^{*} \cdot \Gamma(x,t)^{\frac{1+\mu}{1+p-\mu}} \ \ \text{for all} \ \ (x,t) \in Q_{1/2},
\end{equation*}
where
 \begin{align*}
 \Gamma(x,t):= \left\{
     \begin{aligned}
      & \sup \{r\geq 0: Q_{r}(x,t) \subset \{u>0\}\}, && \text{for} \ \ (x,t) \in \{u>0\},        \\
     & 0,        &&   \text{otherwise}.
     \end{aligned}
     \right.
\end{align*}
\end{corollary}

\begin{proof}
It suffices to prove that
\begin{equation}\label{eq:gradient-decay-dyadic}
L_{1/2^{j+1}}[|Du(x,t)|]
\le
\max\left\{
\widetilde{\mathrm C}_2^*
\left(\frac1{2^j}\right)^{\frac{1+\mu}{1+p-\mu}},
\left(\frac12\right)^{\frac{1+\mu}{1+p-\mu}}
L_{1/2^j}[|Du(x,t)|]
\right\}
\end{equation}
for all $j\ge1$ and some constant
$\widetilde{\mathrm C}_2^*>0$.
Indeed, the conclusion then follows by the dyadic iteration
and localization argument used in the proof of
Theorem~\ref{Sec3:thm1}, with condition \ref{Section6:de1b}) applied to the intrinsic rescalings.   

Assume, by contradiction, that \eqref{eq:gradient-decay-dyadic} fails, then there are solutions $ u_k\in
\mathbb J(F,\lambda_0,\mu,A_0)(Q_1)
\cap\mathcal T_p(Q_1) $ and integers $j_k\ge1$ such that, setting
\[
r_k:=2^{-j_k},
\quad
S_k:=L_{r_k/2}[|Du_k|],
\]
we have
\begin{equation}\label{eq:gradient-decay-failure}
S_k>
\max\left\{
k r_k^{\frac{1+\mu}{1+p-\mu}},
2^{-\frac{1+\mu}{1+p-\mu}}
L_{r_k}[|Du_k|]
\right\}.
\end{equation}

Since $0\le u_k\le1$, condition \ref{Section6:de1a}) yields
\[
S_k
\le L_{1/2}[|Du_k|]
\le
\mathrm C_T
\left(1+\mathcal M_2^{\frac1{1+p}}\right).   
\]    
Combining this uniform bound with
\eqref{eq:gradient-decay-failure}, we obtain
$r_k\to0$, and hence $j_k\to\infty$.
Therefore, the assumption
$\gamma_*:=\liminf_{j\to\infty}\gamma_j>0$
implies $ \gamma_{j_k}\ge\frac{\gamma_*}{2} $ for all sufficiently large $ k $.  

We now define
\[
\tau_k:=\frac{r_k^2}{S_k^p}, 
\quad    \text{and}   \quad 
v_k(x,t):=
\frac{u_k(r_kx,\tau_kt)}{r_kS_k}.
\]
By \eqref{eq:gradient-decay-failure} and the identity
\[
2-\frac{p(1+\mu)}{1+p-\mu}=\theta,
\]
we have
\begin{equation}\label{eq:gradient-decay-time-scale}
0<\tau_k
\le k^{-p}r_k^\theta
\le r_k^\theta.
\end{equation}
By \eqref{eq:gradient-decay-failure}--\eqref{eq:gradient-decay-time-scale}, we have $S_k>0$ and
$\alpha_{j_k}2^{-\theta}
=\tau_k2^{-\theta}
\le 2^{-(j_k+1)\theta}<1$.
Thus condition \ref{Section6:de1c}) applies at the selected scales.

A direct calculation shows that $ v_k $ satisfies
\[
|Dv_k|^p\widetilde F_k(D^2v_k)-(v_k)_t
=
\widetilde\lambda_k(x,t)v_k^\mu
\quad\text{in }Q_1
\]
in the viscosity sense, where
\[
\widetilde F_k(\mathrm{X})
:=
\frac{r_k}{S_k}
F\left(\frac{S_k}{r_k}\mathrm{X}\right), \quad   \text{and}  \quad   \widetilde\lambda_k(x,t)
:=
\frac{r_k^{1+\mu}}{S_k^{1+p-\mu}}
\lambda_0(r_kx,\tau_kt).
\]
It is straightforward to verify that the operator $\widetilde F_k$ has the same ellipticity constants as $F$.

To estimate $v_k$, we use Lemma~\ref{Se3:lem1}
and the dyadic iteration from the proof of
Theorem~\ref{Sec3:thm1}, which yield 
\begin{equation}\label{Section6:coro1:eq900}
L_r[u_k]
\le
\mathrm C_1r^{\frac{2+p}{1+p-\mu}},
\quad 0<r\le\frac12,
\end{equation}
with $\mathrm C_1$ independent of $k$.
For positive times, condition \hyperref[AC]{\bf (AC)} and
Lemma~\ref{Section 2:lemma2} yield
\begin{equation}\label{Section6:coro1:eq901}
u_k(y,t)^{1-\mu}
\le
u_k(y,0)^{1-\mu}+(1-\mu)A_0t.
\end{equation}
Combining \eqref{Section6:coro1:eq900} and \eqref{Section6:coro1:eq901} with time monotonicity
and the definition of $\theta$, we obtain
\begin{equation}\label{eq:gradient-decay-growth}
\sup_{Q_r}u_k
\le
\widehat{\mathrm C}_1
r^{\frac{2+p}{1+p-\mu}},
\quad 0<r\le\frac12,
\end{equation}
where $ \widehat{\mathrm C}_1
:=
\left[
\mathrm C_1^{1-\mu}+(1-\mu)A_0
\right]^{\frac1{1-\mu}} $ is also independent of $ k $.

By \eqref{eq:gradient-decay-growth},
\eqref{eq:gradient-decay-failure}, and
\eqref{eq:gradient-decay-time-scale}, we get 
\begin{equation}\label{eq:gradient-decay-amplitude}
\|v_k\|_{L^\infty(Q_1)}
\le
\frac{
\widehat{\mathrm C}_1
r_k^{\frac{2+p}{1+p-\mu}}
}{r_kS_k}
\le
\frac{\widehat{\mathrm C}_1}{k}.
\end{equation}
Moreover, \eqref{eq:gradient-decay-failure} gives
\[
\|\widetilde\lambda_k\|_{L^\infty(Q_1)}
\le
\frac{\mathcal M_2}{k^{1+p-\mu}}.
\]
Combining this with \eqref{eq:gradient-decay-amplitude},
we find
\begin{equation}\label{eq:gradient-decay-source}
\|\widetilde\lambda_kv_k^\mu\|_{L^\infty(Q_1)}
\le
\frac{
\mathcal M_2\widehat{\mathrm C}_1^\mu
}{k^{1+p}}.
\end{equation}
Applying condition~\ref{Section6:de1c}) to $u_k$ at the scale
$j_k$, and noting that $\alpha_{j_k}=\tau_k$, we obtain  
\[
L_{1/2}[|Dv_k|]
=
\frac1{S_k}
\sup_{B_{r_k/2}\times(-\tau_k2^{-\theta},0]}
|Du_k|
\ge
\gamma_{j_k}
\ge
\frac{\gamma_*}{2}
\]
for all sufficiently large $k$.

On the other hand, condition~\ref{Section6:de1b}) ensures that the gradient estimate applies to the rescaled equation for $v_k$ with the same constant $\mathrm C_T$.
Combining this estimate with
\eqref{eq:gradient-decay-amplitude} and
\eqref{eq:gradient-decay-source}, we obtain    
\[
\begin{aligned}
0<\frac{\gamma_*}{2}
&\le L_{1/2}[|Dv_k|]  \leq
\mathrm C_T
\left(
\|v_k\|_{L^\infty(Q_1)}
+
\|\widetilde\lambda_kv_k^\mu\|_{L^\infty(Q_1)}
^{\frac1{1+p}}
\right)   \\
&\le
\frac{\mathrm C_T}{k}
\left[
\widehat{\mathrm C}_1
+
\left(
\mathcal M_2\widehat{\mathrm C}_1^\mu
\right)^{\frac1{1+p}}
\right]
\longrightarrow 0, \quad   \text{as}  \quad  k \to \infty.
\end{aligned}
\]
This contradiction proves \eqref{eq:gradient-decay-dyadic} and completes the proof.                   
\end{proof}

\begin{remark}
The following stationary profile provides a nontrivial
example in the prescribed solution class $ \mathcal{T}_{p}(Q_{1})$. 
For $p\ge0$ and $0<\mu<1$, take
$F(\mathrm{X})=\operatorname{tr}(\mathrm{X})$ and
\[
\lambda_0
:=
\frac{(2+p)^{1+p}(1+\mu)}
     {(1+p-\mu)^{2+p}}.
\]
Then
\[
u(x,t):=(x_1)_+^{\frac{2+p}{1+p-\mu}}
\]
is a viscosity solution to \eqref{DCP} satisfying
$0\le u\le1$ in $Q_1$, $u(0,0)=0$, and $u_t=0$.
In particular, condition \hyperref[AC]{\textbf{(AC)}}
holds with $A_0=0$. Every admissible translated intrinsic rescaling $v$ of $u$ remains nonnegative, convex in space,
and independent of time.   
By convexity and nonnegativity,
\[
\|Dv\|_{L^\infty(Q_{1/2})}
\le 2\|v\|_{L^\infty(Q_1)}.
\]
Thus conditions \ref{Section6:de1a}) and \ref{Section6:de1b}) hold uniformly with
$\mathrm C_T=2$. To verify condition \ref{Section6:de1c}), note that
\[
|Du(x,t)|
=
\frac{2+p}{1+p-\mu}
(x_1)_+^{\frac{1+\mu}{1+p-\mu}}
\]
does not depend on time.
Since the two suprema in \ref{Section6:de1c}) are taken over the
same spatial ball, we obtain, for every $j\ge1$ satisfying $\alpha_j2^{-\theta}\le1$,    
\[
\begin{aligned}
\sup_{B_{2^{-(j+1)}}\times(-\alpha_j2^{-\theta},0]}
|Du(x,t)|
&=
\sup_{B_{2^{-(j+1)}}}|Du|=
\frac{2+p}{1+p-\mu}
\,2^{-(j+1)\frac{1+\mu}{1+p-\mu}}\\
&=
L_{2^{-(j+1)}}[|Du|]>0.
\end{aligned}
\]
Hence condition \ref{Section6:de1c}) holds with $\gamma_j=1$
for every $j\ge1$, so that $\gamma_*=1$.

Consequently, for this choice of data, we have 
\[
u\in
\mathbb J(F,\lambda_0,\mu,0)(Q_1)
\cap\mathcal T_p(Q_1),
\quad u\not\equiv0.
\]   
\end{remark}

\begin{remark}
Compared to Theorem \ref{Thm3}, we adopt the dyadic iteration scheme used in the proof of Theorem~\ref{Sec3:thm1} and use the normalized decay estimate established in Section~\ref{Section 3}. This approach eliminates the need to restrict $ \mu \in \left(0,\frac{1}{1+p}\right]$ and dispenses with the gradient estimate provided in Proposition \ref{Se2:Pro1} and Remark~\ref{Sec2:rmk2}(i). However, here it is necessary to impose constraints on the class of solutions.
\end{remark}

In conclusion, we will address an improved regularity estimate close to free boundary
points by making use of a Harnack-type inequality for viscosity solutions to \eqref{DCP} in $t-$slices of the parabolic domain. Such a tool plays a decisive role in obtaining sharp estimates close to free boundary points (cf. Choe-Weiss's manuscript \cite[Proposition 3.2]{CW03}).

\begin{theorem}[{\it Sharp regularity in temporal levels}]
\label{Sec6:thm6.1}
Let $u$ be a nonnegative and bounded viscosity solution to  \eqref{DCP} in $ B_{1} \times (0, T) $. Suppose that $ F$ satisfies \hyperref[F1]{\bf (F1)} and $ 0 \leq \partial_{t} u \leq c_{0}(x,t)u^{\mu} $ for $ c_{0} $ a nonnegative bounded function(in the viscosity sense). Then, for any $\tau>0$, there exists a positive constant $\mathrm{C} = \mathrm{C}(n,\lambda, \Lambda, p,\mu,\tau,\mathcal{M}_2, ||c_0||_{\infty}, ||u||_{\infty})>0$
such that for every \(x_{0}\in B_1\) satisfying \(B_{\tau}(x_{0})\subset B_1\) and every \(0<r\leq \tfrac{\tau}{2}\), the following estimate holds: fix $t_0\in(0,T)$, for all $ t \in ( t_{0} - r^{\theta}, t_{0}] \cap(0,T)$,  
\begin{equation}\label{eq:local_sup_bound}
\sup_{B_{r}(x_{0})} u(\cdot, t)
\;\leq\;
\mathrm{C} \max\!\left\{
\inf_{B_{r}(x_{0})} u(\cdot, t),\, r^{\frac{2+p}{1+p-\mu}}
\right\}.
\end{equation}
\end{theorem}

\begin{proof}
We introduce the rescaled function
\begin{equation*}
u_{r}(x,t):=  \frac{u(x_{0}+rx,\, t_{0}+r^{\theta}t)}{r^{\frac{2+p}{1+p-\mu}}},
\end{equation*}
and a straightforward computation shows that $u_{r}$ solves
\begin{equation*}
|Du_{r}|^{p}\, \widetilde{F}(D^{2}u_{r}) - (u_{r})_{t}
= \widetilde{\lambda}_{0}(x,t)\, u_{r}^{\mu}\, \chi_{\{u_{r}>0\}}
\quad \text{in } \quad B_1\times
\left(-\frac{t_0}{r^\theta},
      \frac{T-t_0}{r^\theta}\right),
\end{equation*}
where
\begin{equation*}
\widetilde{\lambda}_{0}(x,t):= \lambda_{0}(x_{0}+rx,\, t_{0}+ r^{\theta}t),
\qquad  \text{and} \qquad
\widetilde{F}(\mathrm{X}):= r^{\frac{p-2\mu}{1+p-\mu}} F\!\left(r^{\frac{2\mu-p}{1+p-\mu}}\mathrm{X}\right).
\end{equation*}
It is immediate that $\widetilde{F}$ still satisfies \hyperref[F1]{\textbf{(F1)}}. The time condition is preserved under this
scaling. By Lemma \ref{Section 2:lemma3}, $ u_{r} $ satisfies on each time slice,
\begin{equation*}
0 \leq |Du_{r}|^{p}\, \widetilde{F}(D^{2}u_{r}) \leq (\mathcal{M}_2+||c_0||_{\infty})u_{r}^{\mu} \quad \text{in} \quad  B_{1},
\end{equation*}
in the viscosity sense.
Invoking the elliptic Harnack argument in \cite[Theorem 8.3]{BJrDaSR2023} (and \cite{CC95} when $ p=0$), we have that for every
$t\in\left(-t_0/r^\theta,(T-t_0)/r^\theta\right)$, 
\begin{equation}
\label{Sec6:thm6.1:eq1}
\sup_{B_{1/2}} u_{r}(\cdot,t)
\leq \mathrm{C}_{0}\Big\{
\inf_{B_{1/2}} u_{r}(\cdot,t)
+ (\mathcal{M}_2+||c_0||_{\infty})^{\frac{1}{1+p}}
\big(\sup_{B_{1}} u_{r}(\cdot,t)\big)^{\frac{\mu}{1+p}}
\Big\},
\end{equation}
where $ \mathrm{C}_{0} \geq 1$ is a positive constant depending only on $ n, \lambda, \Lambda, p$.

Fix $t\in(t_0-r^\theta,t_0]\cap(0,T)$.
We choose  
\[
\frac{\tau}{4}\le c\le\frac{\tau}{2},
\quad r_k:=\frac{c}{2^k}, \quad   0\le k\le n_0,
\quad r_{n_0}=r,
\]
where \(n_0\ge0\) is an integer.
Applying \eqref{Sec6:thm6.1:eq1} at the outer radius
\(r_{k-1}=2r_k\) and the normalized time
\(s=(t-t_0)/r_{k-1}^\theta\), which belongs to the rescaled time interval since $0<t<T$, we obtain
\begin{equation}\label{eq:recurrence}
\begin{aligned}
r_k^{-\frac{2+p}{1+p-\mu}}\sup_{B_{r_k}(x_0)}u(\cdot,t)
\le \mathrm C_{0}\bigg[
&r_k^{-\frac{2+p}{1+p-\mu}}\inf_{B_{r_k}(x_0)}u(\cdot,t)\\
&+\mathrm C_*
\left(r_{k-1}^{-\frac{2+p}{1+p-\mu}}
\sup_{B_{r_{k-1}}(x_0)}u(\cdot,t)\right)^{\frac{\mu}{1+p}}
\bigg],
\end{aligned}
\end{equation}
for \(1\le k\le n_0\), where
\[
\mathrm C_*:=
\max\left\{
1,\,
2^\frac{2+p}{1+p-\mu}
\left(\mathcal{M}_2+\|c_0\|_{\infty}\right)^{\frac1{p+1}}
\right\}.
\]
Since \(r_0=c\ge\tau/4\), we have
\[
r_0^{-\frac{2+p}{1+p-\mu}}\sup_{B_{r_0}(x_0)}u(\cdot,t)
\le
(4/\tau)^\frac{2+p}{1+p-\mu}\|u\|_{\infty}.
\]
If \(n_0=0\), this estimate directly gives \eqref{eq:local_sup_bound}. We may therefore assume \(n_0\ge1\) and distinguish two cases.    

\smallskip
\noindent
\textbf{Case 1.}
For $1 \leq k \leq n_{0}$, we have
\begin{equation*}
r_{k}^{-\frac{2+p}{1+p-\mu}}
\inf_{B_{r_{k}}(x_{0})} u(\cdot,t)
\leq
\mathrm{C}_{*}
\left(
r_{k-1}^{-\frac{2+p}{1+p-\mu}}
\sup_{B_{r_{k-1}}(x_{0})} u(\cdot,t)
\right)^{\frac{\mu}{1+p}}.
\end{equation*}
By \eqref{eq:recurrence}, it reads
\[
r_k^{-\frac{2+p}{1+p-\mu}}\sup_{B_{r_k}(x_0)}u(\cdot,t)
\le 2\mathrm C_0\mathrm C_*
\left(r_{k-1}^{-\frac{2+p}{1+p-\mu}}
\sup_{B_{r_{k-1}}(x_0)}u(\cdot,t)\right)^{\frac{\mu}{1+p}}.
\]
Iterating this inequality yields
\[
\begin{aligned}
 r_{n_0}^{-\frac{2+p}{1+p-\mu}}
 \sup_{B_{r_{n_0}}(x_0)}u(\cdot,t) & \leq
(2\mathrm C_0\mathrm C_*)^{1+\frac{\mu}{1+p}+\cdots+(\frac{\mu}{1+p})^{n_0-1}}
\left(r_0^{-\frac{2+p}{1+p-\mu}}
\sup_{B_{r_0}(x_0)}u(\cdot,t)\right)^{(\frac{\mu}{1+p})^{n_0}}    \\
& \leq
(2\mathrm C_0\mathrm C_*)^{\frac{1+p}{1+p-\mu}}
\max\left\{1,\,
(4/\tau)^{\frac{2+p}{1+p-\mu}}\|u\|_{\infty}
\right\}.
\end{aligned}
\]
Here we used \(0< \frac{\mu}{1+p}<1\) and
\(a^\eta\le\max\{1,a\}\) for \(a\ge0\) and \(0<\eta\le1\).
Since \(r_{n_0}=r\), this gives
\[
\sup_{B_r(x_0)}u(\cdot,t)\le \mathrm C r^\frac{2+p}{1+p-\mu},
\]
and hence proves \eqref{eq:local_sup_bound} in this case.

\smallskip
\noindent
\textbf{Case 2.}
Assume that for some $1 \leq k_{0} \leq n_{0}$,
\begin{equation*}
r_{k_{0}}^{-\frac{2+p}{1+p-\mu}}
\inf_{B_{r_{k_{0}}}(x_{0})} u(\cdot,t)
\geq
\mathrm{C}_{*}
\left(
r_{k_{0}-1}^{-\frac{2+p}{1+p-\mu}}
\sup_{B_{r_{k_{0}-1}}(x_{0})} u(\cdot,t)
\right)^{\frac{\mu}{1+p}},
\end{equation*}
Choose \(k_0\) to be the largest such index. Then the inequality in
{\bf Case~1} holds for \(k_0<k\le n_0\).
At the scale \(r_{k_0}\), \eqref{eq:recurrence} yields
\[
\sup_{B_{r_{k_0}}(x_0)}u(\cdot,t)
\le 2\mathrm C_0
\inf_{B_{r_{k_0}}(x_0)}u(\cdot,t).
\]
Since \(r_{n_0}\le r_{k_0}\), the balls are nested. Therefore,
\[
\begin{aligned}
r_{n_0}^{-\frac{2+p}{1+p-\mu}}
\sup_{B_{r_{n_0}}(x_0)}u(\cdot,t)  & \leq
2\mathrm C_0 r_{n_0}^{-\frac{2+p}{1+p-\mu}}
\inf_{B_{r_{k_0}}(x_0)}u(\cdot,t) \\
& \leq
2\mathrm C_0 r_{n_0}^{-\frac{2+p}{1+p-\mu}}
\inf_{B_{r_{n_0}}(x_0)}u(\cdot,t).
\end{aligned}
\]
Hence \eqref{eq:local_sup_bound} also follows in this case.
\end{proof}

\begin{remark}
In particular, at every free boundary point
$(x_0,t_0)\in\partial\{u>0\}\cap\Omega_T$, we have
\[
\sup_{Q_r(x_0,t_0)}u
\le \mathrm C r^{\frac{2+p}{1+p-\mu}}
\]
for all
$0<r<\min\left\{
1,\frac{\operatorname{dist}_p
((x_0,t_0),\partial\Omega_T)}{2}
\right\}$,
where $\mathrm C$ is independent of $r$. Indeed, continuity gives $u(x_0,t_0)=0$.
By the monotonicity of $u$ in time and
Lemma~\ref{Section 2:lemma2}, we obtain  
\[
u(x_0,t)
\le
\big[(1-\mu)\|c_0\|_\infty(t-t_0)_+\big]^{\frac1{1-\mu}}
\le
\big[(1-\mu)\|c_0\|_\infty\big]^{\frac1{1-\mu}}
r^{\frac{2+p}{1+p-\mu}}
\]
whenever $|t-t_0|<r^\theta$.
After localizing in a fixed spatial neighborhood of $x_0$,
we apply Theorem~\ref{Sec6:thm6.1} at each such time and use
$\inf_{B_r(x_0)}u(\cdot,t)\le u(x_0,t)$
to obtain the desired spatial bound. Taking the supremum over
$|t-t_0|<r^\theta$, we get the asserted estimate. This argument provides an alternative proof of the upper
growth estimate in Theorem~\ref{Thm1} via the degenerate
elliptic Harnack inequality and may be of independent interest.
\end{remark}

\subsection{Future directions} In this subsection, we remark that over the last decade, the regularity theory for nonlinear equations involving gradient-degenerate terms has received considerable attention; see, for instance, \cite{WYJ24, WJ24, WJ26a, WJ26b, ART15, DaSRic2020, BJrDaSR2023, CL13} and the references therein. More recently, dead-core phenomena in elliptic equations have attracted increasing interest due to their connections with reaction-diffusion processes, including (degenerate or singular) fully nonlinear systems \cite{AT24, WJ25}, infinity-Laplacian type equations \cite{ALT16, daSSS25, JSNS24}, $p$-Laplacian equations \cite{daSRS19}, non-local equations \cite{d0TU26}, and the Grad-Mercier equation \cite{CFR25}.    

In conclusion, we propose three open problems that may be of independent interest.

\begin{question}
Consider a fully nonlinear degenerate parabolic system with strong absorption:
\begin{equation*}
\left\{
     \begin{aligned}
     & |Du|^{p} F(D^{2}u) - u_{t} = (v_{+})^{\lambda_{1}} \quad && \text{in} \quad \Omega_{T},   \\
     & |Dv|^{q} G(D^{2}v) - v_{t} = (u_{+})^{\lambda_{2}} \quad && \text{in} \quad \Omega_{T},
     \end{aligned}
\right.
\end{equation*}
where $0 \leq p, q < \infty$, and the parameters $\lambda_{1}, \lambda_{2}$ satisfy suitable structural conditions. Can the results of \cite{AT24, SLR21} be extended to this framework? Even in the case $p = q = 0$, the corresponding theory remains unexplored.
\end{question}

\begin{question}
Consider a degenerate parabolic problem governed by the infinity-Laplacian with strong absorption:
\begin{equation*}
  \Delta_{\infty}^{h} u - u_{t} = \lambda_{0}(x,t) u^{\mu}\chi_{\{u>0\}}(x,t) \quad \text{in} \quad \Omega_{T} := \Omega \times (0,T),
\end{equation*}
where
\begin{equation*}
  \Delta_{\infty}^{h} u := |Du|^{h-3} \Delta_{\infty} u, \quad h \geq 1.
\end{equation*}
When $h = 3$, this model can be interpreted as the parabolic analogue of the framework introduced by Ara\'{u}jo et al.~\cite{ALT16}. It remains an open question whether one can derive analogous improved regularity estimates for more general parabolic equations of $\infty$-Laplacian-type.
\end{question}

\begin{question} Finally, inspired by~\cite{DDS25}, consider the degenerate parabolic problem driven by the normalized $q$-Laplacian with a Hamiltonian term and strong absorption:
\begin{equation*}
  |Du|^{p}\!\left(\Delta_{q}^{\mathrm{N}} u + \mathcal{B}(x,t)\!\cdot\! Du\right) - u_{t}
  = \lambda_{0}(x,t)\,u^{\mu}\chi_{\{u>0\}}(x,t)
  \quad \text{in } \quad \Omega_{T} := \Omega \times (0,T),
\end{equation*}
where $p \ge 0$ and $ 0 < \mu < 1 $, and 
\begin{equation*}
  \Delta_{q}^{\mathrm{N}} u(x,t)
  := \Delta u(x,t) + (q-2)\,\Delta_{\infty}^{\mathrm{N}} u(x,t),
  \quad q \ge 2,
\end{equation*}
is the normalized $q$-Laplacian. Here
\[
\Delta_{\infty}^{\mathrm{N}} u
= \frac{Du^{\mathrm{T}} D^{2}u\,Du}{|Du|^{2}}
\]
denotes the normalized $\infty$-Laplacian.

These considerations naturally raise the question of whether the results established in this manuscript extend to such a broad class of degenerate normalized parabolic models. Furthermore, one may ask which additional analytical tools are required to develop a comprehensive dead-core theory for these equations.
\end{question}

Therefore, these problems appear to be promising directions for future research.

\vspace{3mm}

\appendix

\section*{Appendix A: Proof of Lemma \ref{Se2:lem1}}
\setcounter{equation}{0}
\renewcommand{\theequation}{A.\arabic{equation}}
\label{Appendix:A}

In this appendix, we are devoted to presenting the proof of Lemma  \ref{Se2:lem1}.

\begin{proof}[{\bf Proof of Lemma \ref{Se2:lem1}}]
It suffices to prove comparison on every shorter time cylinder.
We work on such a cylinder and relabel its terminal time as $T$.
By continuity up to the parabolic boundary, the original functions
are bounded on this closed cylinder. Without loss of generality, we assume that $u$ is a strict
viscosity subsolution, i.e.,
\[
  \widetilde{\Phi}(|Du|)F(D^2u)
  -\partial_tu-\lambda_0(x,t)f(u)
  \ge g(x,t)+\delta
  \quad\text{in }\Omega_T,
\]
for some constant $\delta>0$.
Indeed, if this is not the case, we consider the perturbed
function $ \omega(x,t):=u(x,t)-\frac{\varepsilon}{T-t}
  $, for some $ \varepsilon>0 $. Suppose that $\varphi\in C^{2,1}(\Omega_T)$ touches $\omega$ from above at $(x_0,t_0)$. Then $ \psi(x,t):=\varphi(x,t)+\frac{\varepsilon}{T-t} $ touches $u$ from above at the same point.
Since $u$ is a viscosity subsolution, it follows that
\begin{align}\label{Appendix A:eq900}
\begin{split}
0
&\le
\widetilde{\Phi}(|D\psi(x_0,t_0)|)
F(D^2\psi(x_0,t_0))
-\partial_t\psi(x_0,t_0) -\lambda_0(x_0,t_0)f(u(x_0,t_0))
-g(x_0,t_0) \\
&=
\widetilde{\Phi}(|D\varphi(x_0,t_0)|)
F(D^2\varphi(x_0,t_0))
-\partial_t\varphi(x_0,t_0)
-\frac{\varepsilon}{(T-t_0)^2} -\lambda_0(x_0,t_0)f(u(x_0,t_0))
-g(x_0,t_0).
\end{split}
\end{align}
Moreover, since $\omega\le u$, $f$ is nondecreasing,
and $\lambda_0\ge0$, we have
\begin{equation}\label{Appendix A:eq901}
  \lambda_0(x_0,t_0)
  \bigl[f(u(x_0,t_0))-f(\omega(x_0,t_0))\bigr]
  \ge0.
\end{equation}
Combining \eqref{Appendix A:eq900} and \eqref{Appendix A:eq901}, we obtain
\begin{align*}
&\widetilde{\Phi}(|D\varphi(x_0,t_0)|)
F(D^2\varphi(x_0,t_0))
-\partial_t\varphi(x_0,t_0) \\
&\qquad
-\lambda_0(x_0,t_0)f(\omega(x_0,t_0))
-g(x_0,t_0) \\
&\ge
\frac{\varepsilon}{(T-t_0)^2}
+\lambda_0(x_0,t_0)
\bigl[f(u(x_0,t_0))-f(\omega(x_0,t_0))\bigr] \\
&\ge \frac{\varepsilon}{T^2}=:\delta>0.
\end{align*}
Thus, $\omega$ is a strict viscosity subsolution
of \eqref{G-DCP}. Since $\omega\le u\le v$ on $\partial_{\mathrm{par}}\Omega_T$,
it suffices to prove $\omega\le v$ for every $\varepsilon>0$
and then let $\varepsilon \to 0$.
In the remainder of the proof, we denote $\omega$ again by $u$.


We now argue by contradiction. Suppose that  
\[
  \sup_{\Omega_T}(u-v)>0.
\]
Since the renamed function $u$ tends uniformly to $-\infty$
as $t\to T$, while $v$ remains bounded, the boundary
inequality implies that this positive supremum is attained
in $\Omega_T$.

We define the auxiliary function
\begin{equation*}
\Theta_{j}(x,y,t,s) := u(x,t) - v(y,s) - \Psi_{j}(x,y,t,s),
\end{equation*}
where
\begin{equation*}
\Psi_{j}(x,y,t,s) := \frac{j}{l}|x-y|^{l} + \frac{j}{2}(t-s)^{2},
\quad \text{with} \quad l > \max \Big\{2, \frac{2+p}{1+p}\Big\}.
\end{equation*}

Let $(x_j,y_j,t_j,s_j)$ be a maximum point of $\Theta_j$ on $ \overline{\Omega}\times\overline{\Omega}
\times[0,T)\times[0,T]$. Such a point exists because $u$ tends uniformly to $-\infty$ as $t\to T$ and $v$ is bounded.
Comparing with a fixed positive diagonal value and using
the boundedness of the original functions, we obtain
\[
\frac{\varepsilon}{T-t_j}
+\Psi_j(x_j,y_j,t_j,s_j)\le \mathrm{C},
\]
where $\mathrm{C}$ is independent of $j$.
Thus $t_j$ stays away from $T$, while
$|x_j-y_j|+|t_j-s_j|\to0$.
After passing to a subsequence, we therefore have
\[
(x_j,y_j,t_j,s_j)
\longrightarrow
(\hat{x},\hat{x},\hat{t},\hat{t}).
\]
By continuity and the boundary inequality,
$(\hat{x},\hat{t})$ is a positive maximum point of $u-v$
in $\Omega_T$. Consequently, it holds that  $(x_j,y_j,t_j,s_j)\in\Omega\times\Omega\times(0,T)^2 $ for all sufficiently large $ j $.   

We now show that $x_{j} \neq y_{j}$ for $ j \gg 1$. Indeed, if $x_{j} = y_{j}$, next we want to arrive at a contradiction. By the choice of $(x_{j}, y_{j}, t_{j}, s_{j})$ we obtain
\begin{equation*}
u(x_{j}, t_{j}) - v(y_{j}, s_{j}) - \Psi_{j}(x_{j},y_{j},t_{j},s_{j})
\geq u(x,t) - v(y_{j}, s_{j}) - \Psi_{j}(x,y_{j},t,s_{j}).
\end{equation*}
Let
\begin{equation*}
\phi(x,t) := u(x_{j}, t_{j}) - \Psi_{j}(x_{j},y_{j},t_{j},s_{j}) + \Psi_{j}(x,y_{j},t,s_{j}),
\end{equation*}
so that $u(x,t) - \phi(x,t)$ attains a local maximum at $(x_{j}, t_{j})$. A straightforward computation yields
\begin{align*}
\partial_{t}\phi &= j(t-s_{j}),
\qquad D\phi = j|x-y_{j}|^{l-2}(x-y_{j}),\\
D^{2}\phi(x,t) &= j(l-2)|x-y_{j}|^{l-2}\frac{x-y_{j}}{|x-y_{j}|} \otimes \frac{x-y_{j}}{|x-y_{j}|}
+ j|x-y_{j}|^{l-2}\mathbf{I}{\rm{d}_{n}}.
\end{align*}
By the uniform ellipticity of $F$ and $F(\mathrm{O}_n) = 0$\,(\hyperref[F1]{\bf (F1)}), we have
\begin{equation*}
F(D^{2}\phi(x,t_{j})) \leq \Lambda j(l+n-2)|x-y_{j}|^{l-2}.
\end{equation*}
Since $u$ is a strict viscosity subsolution of \eqref{G-DCP}, it follows that
\begin{align*}
0 < \delta  &  \leq   \lim_{\substack{(x,t) \to (x_{j}, t_{j})\\ x \neq x_{j}=y_{j}}}
\bigg (\widetilde{\Phi}(|D \phi(x,t)|)  F(D^{2}\phi(x,t)) - \partial_{t} \phi(x,t) - \lambda_{0}(x,t) f(u(x,t)) -  g(x,t)       \bigg ) \\
& \leq \lim_{\substack{(x,t) \to (x_{j}, t_{j})\\ x \neq x_{j}=y_{j}}}
\bigg ( \mathrm{C}\Lambda(l+n-2) \big\{j^{1+p}|x-y_{j}|^{(l-1)p+l-2}    \big\} - j(t - s_{j}) - \lambda_{0}(x,t) f(u(x,t)) - g(x,t) \bigg ).  
\end{align*}   
Since $(l-1)p + l - 2 > 0$ by definition of $l$, we obtain
\begin{equation}\label{Se2:eq6}
 j(t_{j}-s_{j}) \leq -\lambda_{0}(x_{j},t_{j}) f(u(x_{j},t_{j})) - g(x_{j}, t_{j})- \delta.
\end{equation}
Analogously, one also finds that
\begin{equation}\label{Se2:eq7}
j(t_{j}-s_{j}) \geq - \lambda_{0}(x_{j},s_{j}) f(v(x_{j},s_{j})) - \widetilde{g}(x_{j}, s_{j}).
\end{equation}
Now we combine \eqref{Se2:eq6} and \eqref{Se2:eq7} to obtain
\begin{align*}
 \delta & \leq    \lambda_{0}(x_{j},s_{j}) f(v(x_j, s_j)) -  \lambda_{0}(x_{j},t_{j})f(u(x_j, t_j))  + \widetilde{g}(x_{j}, s_{j})- g(x_{j}, t_{j})  \notag    \\
 & =\lambda_{0}(x_j, t_j) [f(v(x_j,s_j))-f(u(x_j,t_j))]  + [\lambda_0(x_j, s_j)-\lambda_0(x_j,t_j)]f(v(x_j,s_j)) + \widetilde{g}(x_{j}, s_{j})- g(x_{j}, t_{j})  \notag \\
 & \leq  |\lambda_0(x_j,s_j)-\lambda_0(x_j,t_j)||f(v(x_j,s_j))|  + |g(x_{j}, s_{j})- g(x_{j}, t_{j})|  \to  0   \quad  \text{as}  \quad  j \to \infty,
\end{align*}
where we have used the fact $ u(x_j,t_j) > v(x_j,s_j) $, $ \widetilde{g}(x_j, s_j) \leq g(x_j, s_j) $, the monotonicity of $ f $, and $ \lambda_0 \geq 0 $ in the second inequality, which contradicts with $ \delta >0 $. Therefore, we conclude that $ x_{j} \neq y_{j} $ for $ j \gg 1 $, as claimed.

By Ishii-Lions Lemma \cite{CIL92}, there exist $ \mathrm{X_{j}}, \mathrm{Y_{j}} \in \mathrm{Sym}(n) $ such that
\begin{align}\label{Se2:eq1}
\begin{split}
& (\partial_{t} \Psi_{j}, D_{x} \Psi_{j}, \mathrm{X_{j}}) \in \overline{\mathcal{P}}^{2,+} u(x_{j}, t_{j}),  \\
& (-\partial_{s} \Psi_{j}, -D_{y} \Psi_{j}, \mathrm{Y_{j}})  \in \overline{\mathcal{P}}^{2,-} v(y_{j}, s_{j})
\end{split}
\end{align}
and the following matrix inequality holds
\begin{equation*}
\begin{pmatrix}
\mathrm{X_{j}}  &   0   \\
0  &    -\mathrm{Y_{j}}
\end{pmatrix}
\leq
\begin{pmatrix}
\mathrm{Z}   &   -\mathrm{Z}   \\
-\mathrm{Z}  &    \mathrm{Z}
\end{pmatrix}
 + \frac{2}{\mu_0}
 \begin{pmatrix}
\mathrm{Z}^{2}   &   -\mathrm{Z}^{2}   \\
-\mathrm{Z}^{2}  &    \mathrm{Z}^{2}
\end{pmatrix}, \ \  0 < \mu_0 \ll 1,
\end{equation*}
with
\begin{equation*}
 \mathrm{Z}:= D_{xx}^{2}\Psi_{j}(x_j, y_j, t_j, s_j)  \quad  \text{and} \quad \mathrm{Z}^{2}:= j^{2}|x_{j}-y_{j}|^{2l-4} \textbf{I}{\rm{d}_{n}} + l(l-2) j^{2}|x_{j}-y_{j}|^{2l-6} (x_{j}-y_{j}) \otimes (x_{j}-y_{j}) .   
\end{equation*}
From \cite[Theorem 3.2]{DDS25}, we see that $ \mathrm{X_{j}} \leq \mathrm{Y_{j}} $. Before two viscosity inequalities are given, we denote
\begin{equation}\label{Se2:eq2}
  \eta_{j}:= D_{x} \Psi_{j} = - D_{y} \Psi_{j} = j|x_{j}-y_{j}|^{l-2}(x_{j}-y_{j}).
\end{equation}
We combine \eqref{Se2:eq1} and \eqref{Se2:eq2} to get
\begin{equation*}
  \widetilde{\Phi}(|\eta_{j}|)F(\mathrm{X_{j}}) - \partial_{t} \Psi_{j} - \lambda_{0}(x_{j},t_{j}) f(u(x_{j},t_{j})) - g(x_{j},t_{j}) \geq \delta   \quad \text{in}  \quad  \Omega_{T},
\end{equation*}
and
\begin{equation*}
  \widetilde{\Phi}(|\eta_{j}|)F(\mathrm{Y_{j}}) + \partial_{s} \Psi_{j} - \lambda_{0}(y_{j},s_{j}) f(v(y_{j},s_{j})) - \widetilde{g}(y_{j},s_{j}) \leq 0 \quad \text{in} \quad  \Omega_{T}.
\end{equation*}
Subtracting these two viscosity inequalities above, it reads
\begin{align*}
 \widetilde{\Phi}(|\eta_{j}|)(F(\mathrm{X_{j}}) - F(\mathrm{Y_{j}))} & +  \lambda_{0}(y_{j},s_{j}) f(v(y_{j},s_{j})) - \lambda_{0}(x_{j},t_{j}) f(u(x_{j},t_{j})) +(- \partial_{t} \Psi_{j} - \partial_{s} \Psi_{j})    \\
 & + \widetilde{g}(y_{j},s_{j}) - g(x_{j}, t_{j})  \\
 & := D_{1}  + D_{2}  + D_{3}  + D_{4}  \geq \delta > 0.
\end{align*}

The remaining part mainly involves estimating four terms $ D_{i} $, $ i=1,2,3,4$.
\vspace{1mm}

{\boxed{\text{The estimate of} \ D_{1}.}} In view of $ \mathrm{X_{j}} \leq \mathrm{Y_{j}} $ and \hyperref[F1]{\bf (F1)}, it follows that $ D_{1} \leq 0 $.

\vspace{1mm}

{\boxed{\text{The estimate of} \ D_{2}.}} Noticed that $ (x_{j}, y_{j}, t_{j}, s_{j}) \rightarrow  (\widehat{x}, \widehat{x}, \widehat{t}, \widehat{t}) $ as $ j \rightarrow \infty $, the continuity of $ \lambda_{0} $,  $ \lambda_{0} > 0 $, $ u(\widehat{x}, \widehat{t}) > v(\widehat{x}, \widehat{t})$ and monotonicity of $ f$, hence we get $ \limsup_{j \rightarrow \infty } D_{2} \leq 0 $

\vspace{1mm}

{\boxed{\text{The estimate of} \ D_{3}.}} A direct calculation yields $ D_{3} = 0 $.

\vspace{1mm}

{\boxed{\text{The estimate of} \ D_{4}.}} Since $ \widetilde{g} \leq g $ in $ \Omega_{T} $, then $ \limsup_{j \rightarrow \infty} D_{4} \leq 0 $.

\vspace{1mm}

Consequently, it concludes
\begin{equation*}
  0< \delta \leq \limsup_{j \rightarrow \infty} (D_{1} + D_{2}  +  D_{3} + D_{4} ) \leq 0, 
\end{equation*}
which is a contradiction. This closes the proof of this lemma.
\end{proof}

\vspace{2mm}

{
\section*{Appendix B: Proof of Lemmas~\ref{Section 2:lemma2} and \ref{Section 2:lemma3}}
\label{Appendix:B}
\setcounter{equation}{0}
\renewcommand{\theequation}{B.\arabic{equation}}
\renewcommand{\theHequation}{B.\arabic{equation}}

\begin{proof}[{\bf Proof of Lemma~\ref{Section 2:lemma2}}]
Fix $x_*\in B_1 $ and write
\[
z(t)=u(x_*,t).
\]
We first transfer condition~\hyperref[AC]{\bf (AC)} to $z$. We then derive the increment estimates and use them to control
the time derivatives of touching test functions.

\medskip
\noindent
{\bf Step 1.} Let $\eta\in C^1$ touch $z$ from above at an interior time $t_*$.
For a fixed $\delta>0$, the function $ z(t)-\eta(t)-\delta(t-t_*)^2 $ has a strict local maximum at $ t_* $. Choose a small closed cylinder $ K$ around $(x_*,t_*)$, contained in
$ B_1 \times I $. Let $(x_\varepsilon,t_\varepsilon)$ be a maximum
point of
\[
u(x,t)-\eta(t)-\delta(t-t_*)^2
-\frac{|x-x_*|^2}{\varepsilon}
\]
on this cylinder $ K$. Comparing the maximum value with the value at $(x_*,t_*)$ gives
\[
\frac{|x_\varepsilon-x_*|^2}{\varepsilon}\le \mathrm{C}.
\]
Hence $x_\varepsilon\to x_*$. By continuity of $u$ and the strict
time contact, we also have $t_\varepsilon\to t_*$. Thus, the maximum points $(x_\varepsilon,t_\varepsilon)$
lie in the interior of $K$ for all sufficiently small
$\varepsilon$.

After adding a constant, the function
\[
\phi_\varepsilon(x,t)
:=
\eta(t)+\delta(t-t_*)^2
+\frac{|x-x_*|^2}{\varepsilon}+ \mathrm{c}_\varepsilon
\]
touches $u$ from above at $(x_\varepsilon,t_\varepsilon)$.
By the upper inequality in \hyperref[AC]{\bf (AC)}, we have
\[
\eta'(t_\varepsilon)+2\delta(t_\varepsilon-t_*)
\le A_0u(x_\varepsilon,t_\varepsilon)^\mu.
\]
Passing to the limit as $\varepsilon\to0$, we obtain
\[
\eta'(t_*)\le A_0z(t_*)^\mu.
\]

If $\eta\in C^1$ touches $z$ from below at $t_*$,
we instead minimize
\[
u(x,t)-\eta(t)+\delta(t-t_*)^2
+\frac{|x-x_*|^2}{\varepsilon}.
\]
The corresponding test function touching $u$ from below is
\[
\eta(t)-\delta(t-t_*)^2
-\frac{|x-x_*|^2}{\varepsilon}+ \mathrm{c}_\varepsilon.
\]
Using the same convergence argument and the lower inequality
in \hyperref[AC]{\textbf{(AC)}}, we obtain $\eta'(t_*)\ge0$. Therefore,
\[
0\le z'\le A_0z^\mu
\]
in the viscosity sense.

\medskip
\noindent
{\bf Step 2.} We first establish the following elementary claim:
if $y$ is continuous on an open interval $J$ and satisfies
$y'\le L$ in the viscosity sense, where $L$ is a constant, then
\begin{equation}\label{AppB:Lemma2.2:eq1}
y(b)-y(a)\le L(b-a), \quad  \text{whenever} \quad a<b, \quad [a, b] \subset J.
\end{equation}
Suppose, to the contrary, that \eqref{AppB:Lemma2.2:eq1} fails. Choose a constant $ q$ such that   
\[
L<q<\frac{y(b)-y(a)}{b-a}
\]
and fix $ \mathrm{c}>b$ such that $[a,\mathrm{c}]\subset J$. In particular, one has
\[
d:=y(b)-y(a)-q(b-a)>0.
\]
For $\varepsilon>0$, define
\[
H_{\varepsilon}(t)
:=
y(t)-qt-\frac{\varepsilon}{\mathrm{c}-t},
\qquad t\in[a,\mathrm{c}).
\]
A direct computation gives
\[
\begin{aligned}
H_{\varepsilon}(b)-H_{\varepsilon}(a)
&=
d-\varepsilon
\left(\frac{1}{\mathrm{c}-b}-\frac{1}{\mathrm{c}-a}\right)\\
&=
d-\varepsilon\frac{b-a}{(\mathrm{c}-b)(\mathrm{c}-a)} > 0
\end{aligned}
\]
provided $ \varepsilon $ is small enough.
Moreover, since $y$ is bounded on $[a,\mathrm{c}]$, we have
\[
H_\varepsilon(t)\longrightarrow-\infty
\quad\text{as }t\to\mathrm{c}^{-}.
\]
Consequently, $H_{\varepsilon}$ attains its maximum on $[a,\mathrm{c})$.
Since $H_{\varepsilon}(b)>H_{\varepsilon}(a)$, any maximizer
$t_{\varepsilon}$ lies in $(a,\mathrm{c})$.  

Set $\mathrm{C}_{\varepsilon}:=H_{\varepsilon}(t_{\varepsilon})$.
The smooth function
\[
\varphi_{\varepsilon}(t)
:=
qt+\frac{\varepsilon}{\mathrm{c}-t}+ \mathrm{C}_{\varepsilon}
\]
touches $ y $ from above at $t_{\varepsilon}$. Therefore, the viscosity inequality yields
\[
q+\frac{\varepsilon}{(\mathrm{c}-t_{\varepsilon})^{2}}
=
\varphi_{\varepsilon}'(t_{\varepsilon})
\le L,
\]
contradicting $q>L$. This proves the stated claim.  

We now apply this claim to the time trace $z$. Since $z'\ge0$
in the viscosity sense, the function $-z$ satisfies $(-z)'\le0$.
Applying the preceding result with $L=0$ gives
\[
z(b)\ge z(a)
\quad\text{for all }a<b\text{ in }I.
\]
Hence $z$ is nondecreasing.

To obtain the upper increment bound, fix $a<b$ in $I$.
For every sufficiently small $\rho>0$, we have
$[a-\rho,b+\rho]\subset I$. Define
\[
L_{\rho}
:=
A_{0}\max_{t\in[a-\rho,b+\rho]}z(t)^{\mu}.
\]
The upper viscosity inequality for $ z $ implies
\begin{equation*}
z'\le L_{\rho}
\quad\text{in }(a-\rho,b+\rho).
\end{equation*}
Applying \eqref{AppB:Lemma2.2:eq1} with $L=L_\rho$
and then letting $\rho\downarrow0$, we obtain, by continuity,
\[
0\le z(b)-z(a)
\le
A_{0}\max_{t\in[a,b]}z(t)^{\mu}\,(b-a).
\]
This estimate also proves that $z$ is locally Lipschitz and hence
locally absolutely continuous.

\medskip
\noindent
{\bf Step 3.} Dividing the preceding increment estimate by the time
difference and passing to the limit at points where $z$
is differentiable, we obtain
\[
0\le z'(t)\le A_{0}z(t)^{\mu}
\quad \text{for almost every }t\in I.
\]
Finally, fix $h>0$. Since $z\ge0$, the function
$(z+h)^{1-\mu}$ is locally absolutely continuous, and the chain
rule yields 
\[
\begin{aligned}
\frac{d}{dt}(z+h)^{1-\mu}
&=
(1-\mu)(z+h)^{-\mu}z'\\
&\le
A_{0}(1-\mu)
\left(\frac{z}{z+h}\right)^{\mu}\\
&\le A_{0}(1-\mu)
\end{aligned}
\]
almost everywhere in $I$. Integrating over $[a,b]$ gives 
\[
(z(b)+h)^{1-\mu}-(z(a)+h)^{1-\mu}
\le A_{0}(1-\mu)(b-a).
\]
Letting $ h \to 0 $, we conclude that
\[
z(b)^{1-\mu}
\le
z(a)^{1-\mu}+A_{0}(1-\mu)(b-a).
\]
This argument also covers the case $z(a)=0$.

\medskip
\noindent
{\bf Step 4.} Fix an interior point $(x_0,t_0)$. Let $  z(t):= u(x_0,t) $ and set
\[
m:=A_0z(t_0)^\mu.
\]
By {\bf Step 2} and continuity, the difference quotients
\[
q_h:=\frac{z(t_0+h)-z(t_0)}{h}
\]
satisfy
\[
0\le q_h\le m+o(1)
\quad  \text{as } h \to 0,
\]
for both positive and negative $h$.

Let $\phi$ touch $u$ from above at $(x_0,t_0)$.
Restricting the contact inequality to $x=x_0 $, we have
\[
z(t_0+h)-z(t_0)
\le
\phi(x_0,t_0+h)-\phi(x_0,t_0).
\]
For $h>0$, dividing by $h$ and taking the limit yields  
$\phi_t(x_0,t_0)\ge0$. For $h<0$, the inequality reverses after
division, and the upper bound for $q_h$ gives
$\phi_t(x_0,t_0) \le m $. The same bounds follow when $\phi$ touches $u$ from below,
by applying the corresponding contact inequality for
positive and negative $h$.    

Finally, if $u$ is continuous at a terminal time, the two increment
estimates extend to that time by taking limits from interior times.
\end{proof}

Next we give the proof of Lemma~\ref{Section 2:lemma3}.

\begin{proof}[{\bf Proof of Lemma~\ref{Section 2:lemma3}}]
We divide the proof into the following two steps.

\medskip
\noindent
{\bf Step 1.}
By Lemma~\ref{Section 2:lemma2}, every test function
$\varphi\in C^{2,1}$ touching $u$ from above or below
at an interior point $(x_0,t_0)$ satisfies
\[
0\le \varphi_t(x_0,t_0)
\le A_0u(x_0,t_0)^\mu.
\]
If $\varphi$ touches $u$ from above, the viscosity equation gives
\[
|D\varphi(x_0,t_0)|^p
F(D^2\varphi(x_0,t_0))
-\varphi_t(x_0,t_0)
\ge
\lambda_0(x_0,t_0)u(x_0,t_0)^\mu.
\]
Since $\varphi_t(x_0,t_0)\ge0$ and $\lambda_0\ge0$, it follows that
\[
|D\varphi(x_0,t_0)|^p
F(D^2\varphi(x_0,t_0))
\ge0.
\]
If $\varphi$ touches $u$ from below, the opposite viscosity
inequality and the upper bound for $\varphi_t$ yield
\[
\begin{aligned}
|D\varphi(x_0,t_0)|^p
F(D^2\varphi(x_0,t_0))
&\le
\lambda_0(x_0,t_0)u(x_0,t_0)^\mu
+\varphi_t(x_0,t_0)\\
&\le
(\mathcal M_2+A_0)u(x_0,t_0)^\mu.
\end{aligned}
\]

\medskip
\noindent
{\bf Step 2.}
Fix $t_*\in(a,b)$ and write $\omega(x):=u(x,t_*)$.
We verify the two elliptic viscosity inequalities separately.

First, let $\psi\in C^2$ touch $\omega$ from above at
$x_0\in B_1$. Adding a constant to $\psi$ if necessary,
we may assume that
\[
\omega(x_0)=\psi(x_0),
\quad
\omega(x)\le\psi(x)
\quad\text{for }x\in\overline{B_r(x_0)},
\]
where $r>0$ is sufficiently small and
$\overline{B_r(x_0)}\subset B_1$.
Fix $\delta>0$ and set
\[
\psi_\delta(x):=\psi(x)+\delta|x-x_0|^4.
\]
Then
\[
\omega(x)-\psi_\delta(x)
\le -\delta|x-x_0|^4
\quad\text{on } \quad \overline{B_r(x_0)}.
\]
Now, we choose $\tau>0$ such that $[t_*-\tau,t_*+\tau]\subset(a,b)$,
and define the compact cylinder $ \mathcal K:= \overline{B_r(x_0)} \times[t_*-\tau,t_*+\tau] $. For each $\varepsilon>0$, let
$(x_\varepsilon,t_\varepsilon)\in\mathcal K$ be a maximum point of
\[
u(x,t)-\psi_\delta(x)
-\frac{(t-t_*)^2}{\varepsilon}.
\]
The comparison at $(x_0,t_*)$ leads to 
\[
0\le
u(x_\varepsilon,t_\varepsilon)
-\psi_\delta(x_\varepsilon)
-\frac{(t_\varepsilon-t_*)^2}{\varepsilon}.
\]
Since $u$ and $\psi_\delta$ are bounded on the relevant compact
sets, there exists a constant $\mathrm{C}$, independent of
$\varepsilon$, such that
\[
\frac{(t_\varepsilon-t_*)^2}{\varepsilon}\le \mathrm{C}.
\]
Consequently, we obtain $t_\varepsilon\to t_*$.

We next show that $x_\varepsilon\to x_0$.
Since $x_\varepsilon\in\overline{B_r(x_0)}$, every sequence
$\varepsilon_j\to0$ has a subsequence, still denoted by
$\varepsilon_j$, such that $x_{\varepsilon_j}\to\widehat x$
for some $\widehat x\in\overline{B_r(x_0)}$.
By the preceding maximum inequality,
\[
0\le u(x_{\varepsilon_j},t_{\varepsilon_j})
     -\psi_\delta(x_{\varepsilon_j}).
\]
Using $t_{\varepsilon_j}\to t_*$ and continuity, we obtain
\[
0\le \omega(\widehat x)-\psi_\delta(\widehat x)
\le -\delta|\widehat x-x_0|^4.
\]
Hence $\widehat x=x_0$. Since every subsequential limit
equals $x_0$, we conclude that $x_\varepsilon\to x_0$. In particular, $(x_\varepsilon,t_\varepsilon)$ lies in the
interior of $\mathcal K$ for all sufficiently small $\varepsilon$.

Now, we define
\[
\mathrm{c}_\varepsilon
:=
u(x_\varepsilon,t_\varepsilon)
-\psi_\delta(x_\varepsilon)
-\frac{(t_\varepsilon-t_*)^2}{\varepsilon}.
\]
Then the function
\[
\varphi_\varepsilon(x,t)
:=
\psi_\delta(x)
+\frac{(t-t_*)^2}{\varepsilon}
+\mathrm{c}_\varepsilon
\]
touches $u$ from above at
$(x_\varepsilon,t_\varepsilon)$.
The upper-contact inequality established in {\bf Step~1} therefore gives
\[
|D\psi_\delta(x_\varepsilon)|^p
F(D^2\psi_\delta(x_\varepsilon))
\ge0.
\]
Letting $\varepsilon\to0$ and using the continuity of $F$
together with    
\[
D\psi_\delta(x_0)=D\psi(x_0),
\quad \text{and} \quad 
D^2\psi_\delta(x_0)=D^2\psi(x_0),
\]
we obtain
\[
|D\psi(x_0)|^pF(D^2\psi(x_0))\ge0.
\]
This proves the first elliptic viscosity inequality.

To prove the second inequality, let $\psi\in C^2$ touch
$\omega$ from below at $x_0$. Choose $r>0$ as above so that
\[
\omega(x_0)=\psi(x_0),
\quad
\omega(x)\ge\psi(x)
\quad\text{on} \quad \overline{B_r(x_0)}.
\]
For $\delta>0$, now set
\[
\psi_\delta(x)=\psi(x)-\delta|x-x_0|^4,
\]
so that
\[
\omega(x)-\psi_\delta(x)
\ge\delta|x-x_0|^4.
\]
On a compact cylinder $\mathcal K$ chosen as above, let
$(x_\varepsilon,t_\varepsilon)$ minimize
\[
u(x,t)-\psi_\delta(x)
+\frac{(t-t_*)^2}{\varepsilon}.
\]
By the minimality of $(x_\varepsilon,t_\varepsilon)$, we have
\[
u(x_\varepsilon,t_\varepsilon)-\psi_\delta(x_\varepsilon)
+\frac{(t_\varepsilon-t_*)^2}{\varepsilon}
\le u(x_0,t_*)-\psi_\delta(x_0)=0.
\]
Hence, we obtain
\[
\frac{(t_\varepsilon-t_*)^2}{\varepsilon}\le \mathrm{C},
\]
where $\mathrm{C}$ is independent of $\varepsilon$, and hence
$t_\varepsilon\to t_*$.
For any subsequential limit $\widehat x$ of $x_\varepsilon$,
continuity gives
\[
0\ge
\omega(\widehat x)-\psi_\delta(\widehat x)
\ge\delta|\widehat x-x_0|^4.
\]
Therefore, $x_\varepsilon\to x_0$, and the minimum points
are interior for sufficiently small $\varepsilon$.

Now, we define
\[
\mathrm{c}_\varepsilon
:=u(x_\varepsilon,t_\varepsilon)
-\psi_\delta(x_\varepsilon)
+\frac{(t_\varepsilon-t_*)^2}{\varepsilon}.
\]
By the minimality of $(x_\varepsilon,t_\varepsilon)$, the function
\[
\varphi_\varepsilon(x,t)
:=\psi_\delta(x)-\frac{(t-t_*)^2}{\varepsilon}
+\mathrm{c}_\varepsilon
\]
touches $u$ from below at $(x_\varepsilon,t_\varepsilon)$. Applying the lower-contact inequality from {\bf Step~1}, we obtain
\[
|D\psi_\delta(x_\varepsilon)|^p
F(D^2\psi_\delta(x_\varepsilon))
\le
(\mathcal M_2+A_0)
u(x_\varepsilon,t_\varepsilon)^\mu.
\]
Passing to the limit as $\varepsilon \to 0$ gives
\[
|D\psi(x_0)|^pF(D^2\psi(x_0))
\le
(\mathcal M_2+A_0)\omega(x_0)^\mu.
\]
Since $x_0$ and the test functions were arbitrary, both
inequalities hold in the viscosity sense in $B_1$.
This completes the proof of Lemma~\ref{Section 2:lemma3}.
\end{proof}
}

\end{document}